\documentclass[11pt]{amsart}

\usepackage[T1]{fontenc}

\usepackage{charter}
\usepackage{eulerpx}

\usepackage{float, enumitem}

\usepackage{amsmath,amssymb,amsthm,array,color,multirow,pdflscape,graphicx,pigpen,stmaryrd,comment, float, color}
\usepackage[all]{xypic}
\usepackage[normalem]{ulem}

\usepackage{tikz}
\usepackage{tikz-cd}
\tikzset{shorten <>/.style={shorten >=#1,shorten <=#1}}

\usepackage{stackrel}

\usepackage{todonotes}

\usepackage[margin=1.2in]{geometry}

\usepackage[all,arc,2cell]{xy}
\UseAllTwocells

\DeclareMathAlphabet\mathbfcal{OMS}{cmsy}{b}{n}

\expandafter\let\expandafter\oldproof\csname\string\proof\endcsname
\let\oldendproof\endproof
\renewenvironment{proof}[1][\proofname]{%
  \oldproof[\bfseries #1]%
}{\oldendproof}

\usepackage{xcolor}
\definecolor{darkgreen}{rgb}{0,0.50,0} 
\definecolor{darkred}{rgb}{0.75,0,0}
\definecolor{darkblue}{rgb}{0,0,0.6} 
\definecolor{lightblue}{rgb}{.6,.6,0.9} 
\usepackage[pdfborder=0,pagebackref,colorlinks,citecolor=darkgreen,linkcolor=darkgreen,urlcolor=darkblue]{hyperref}
\renewcommand*{\backref}[1]{}
\renewcommand*{\backrefalt}[4]{({%
    \ifcase #1 Not cited.%
          \or On p.~#2%
          \else On pp.~#2%
    \fi%
    })}

\def\makeautorefname#1#2{\expandafter\def\csname#1autorefname\endcsname{#2}}
\makeautorefname{equation}{Equation}%
\makeautorefname{footnote}{footnote}%
\makeautorefname{item}{item}%
\makeautorefname{figure}{Figure}%
\makeautorefname{table}{Table}%
\makeautorefname{part}{Part}%
\makeautorefname{appendix}{Appendix}%
\makeautorefname{chapter}{Chapter}%
\makeautorefname{section}{Section}%
\makeautorefname{subsection}{Section}%
\makeautorefname{subsubsection}{Section}%
\makeautorefname{paragraph}{Paragraph}%
\makeautorefname{subparagraph}{Paragraph}%
\makeautorefname{theorem}{Theorem}%
\makeautorefname{thm}{Theorem}%
\makeautorefname{addm}{Addendum}%
\makeautorefname{mainthm}{Main theorem}%
\makeautorefname{corollary}{Corollary}%
\makeautorefname{cor}{Corollary}%
\makeautorefname{lemma}{Lemma}%
\makeautorefname{lem}{Lemma}%
\makeautorefname{sublemma}{Sublemma}%
\makeautorefname{sublem}{Sublemma}%
\makeautorefname{subl}{Sublemma}%
\makeautorefname{prop}{Proposition}%
\makeautorefname{property}{Property}
\makeautorefname{pro}{Property}
\makeautorefname{sch}{Scholium}%
\makeautorefname{step}{Step}%
\makeautorefname{conject}{Conjecture}%
\makeautorefname{conj}{Conjecture}%
\makeautorefname{questn}{Question}
\makeautorefname{quest}{Question}
\makeautorefname{qn}{Question}
\makeautorefname{definition}{Definition}%
\makeautorefname{defn}{Definition}%
\makeautorefname{defi}{Definition}%
\makeautorefname{def}{Definition}%
\makeautorefname{dfn}{Definition}%
\makeautorefname{notation}{Notation}
\makeautorefname{notn}{Notation}
\makeautorefname{rem}{Remark}%
\makeautorefname{rems}{Remarks}%
\makeautorefname{rmk}{Remark}%
\makeautorefname{rk}{Remark}%
\makeautorefname{remarks}{Remarks}%
\makeautorefname{rems}{Remarks}%
\makeautorefname{rmks}{Remarks}%
\makeautorefname{rks}{Remarks}%
\makeautorefname{example}{Example}%
\makeautorefname{examp}{Example}%
\makeautorefname{exmp}{Example}%
\makeautorefname{exam}{Example}%
\makeautorefname{exa}{Example}%
\makeautorefname{axiom}{Axiom}%
\makeautorefname{axi}{Axiom}%
\makeautorefname{ax}{Axiom}%
\makeautorefname{case}{Case}%
\makeautorefname{claim}{Claim}%
\makeautorefname{clm}{Claim}%
\makeautorefname{assumpt}{Assumption}%
\makeautorefname{asses}{Assumptions}%
\makeautorefname{conclusion}{Conclusion}%
\makeautorefname{concl}{Conclusion}%
\makeautorefname{conc}{Conclusion}%
\makeautorefname{cond}{Condition}%
\makeautorefname{const}{Construction}%
\makeautorefname{con}{Construction}%
\makeautorefname{criterion}{Criterion}%
\makeautorefname{criter}{Criterion}%
\makeautorefname{crit}{Criterion}%
\makeautorefname{exercise}{Exercise}%
\makeautorefname{exer}{Exercise}%
\makeautorefname{exe}{Exercise}%
\makeautorefname{problem}{Problem}%
\makeautorefname{problm}{Problem}%
\makeautorefname{prob}{Problem}%
\makeautorefname{prob}{Problem}%
\makeautorefname{soln}{Solution}%
\makeautorefname{sol}{Solution}%
\makeautorefname{sum}{Summary}%
\makeautorefname{oper}{Operation}%
\makeautorefname{obs}{Observation}%
\makeautorefname{ob}{Observation}%
\makeautorefname{conv}{Convention}%
\makeautorefname{cvn}{Convention}%
\makeautorefname{warn}{Warning}%
\makeautorefname{note}{Note}%
\makeautorefname{fact}{Fact}%
\makeautorefname{ouch0}{Counterexample}%
\makeautorefname{introthm}{Theorem}%
\newtheorem{thm}{Theorem}[section]
\newtheorem{introthm}{Theorem}

\newtheorem{prop}{Proposition}[section]
\newtheorem{lem}{Lemma}[section]

\theoremstyle{definition}

\newtheorem{exmp}{Example}[section]

\newtheorem{rem}{Remark}[section]

\makeatletter
\let\c@cor=\c@thm
\let\c@prop=\c@thm
\let\c@lem=\c@thm
\let\c@conj=\c@thm
\let\c@defn=\c@thm
\let\c@claim=\c@thm
\let\c@quest=\c@thm
\let\c@df=\c@thm
\let\c@exmp=\c@thm
\let\c@rem=\c@thm
\let\c@sch=\c@thm
\let\c@con=\c@thm
\let\c@equation\c@thm
\let\c@introthm=\c@thm
\makeatother

\numberwithin{equation}{section}

\newcommand{\R}{\mathbb R}

\newcommand{\id}{\textup{id}}

\newcommand{\Diff}{\textup{Diff}\,}
\newcommand{\SDiff}{\textup{SDiff}^+}

\newcommand{\Aut}{\textup{Aut}}

\newcommand{\TP}{P}
\newcommand{\TJ}{\mathrm{J}}

\newcommand{\po}{\ar@{}[dr]|{\text{\pigpenfont R}}}
\newcommand{\pb}{\ar@{}[dr]|{\text{\pigpenfont J}}}

\newcommand{\HH}{\mathcal{{H}}}
\newcommand{\dualC}{{\Upsilon}}

\newcommand{\auth}{\mathrm{Aut}(\HH)}

\newcommand{\autxi}{\mathrm{Aut}(\xi)}

\newcommand{\dast}{{\displaystyle{\ast}}}

\newlength{\storeparskip}
\title[Strict contactomorphisms]{DEFORMATION RETRACTION of the GROUP of STRICT CONTACTOMORPHISMS
of the THREE-SPHERE to the UNITARY GROUP}

\author[DeTurck]{Dennis DeTurck}
\address{Department of Mathematics, The University of Pennsylvania}
\email{deturck@math.upenn.edu}

\author[Gluck]{Herman Gluck}
\address{Department of Mathematics, The University of Pennsylvania}
\email{gluck@math.upenn.edu}

\author[Lichtenfelz]{Leandro Lichtenfelz }
\address{Department of Mathematics, Wake Forest University}
\email{llichte2@gmail.com}

\author[Merling]{\\ Mona Merling}
\address{Department of Mathematics, The University of Pennsylvania}
\email{mmerling@math.upenn.edu}

\author[Wang]{Yi Wang}
\address{Department of Mathematics, The University of Pennsylvania}
\email{yiwang4@math.upenn.edu}

\author[Yang]{Jingye Yang}
\address{Department of Mathematics, The University of Pennsylvania}
\email{jingyey@math.upenn.edu}

\let\oldtocsection=\tocsection

\let\oldtocsubsection=\tocsubsection

\let\oldtocsubsubsection=\tocsubsubsection

\renewcommand{\tocsection}[2]{\hspace{0em}\oldtocsection{#1}{#2}}
\renewcommand{\tocsubsection}[2]{\hspace{2em}\oldtocsubsection{#1}{#2}}
\renewcommand{\tocsubsubsection}[2]{\hspace{2em}\oldtocsubsubsection{#1}{#2}}

\begin{document}

\vspace{-15pt}

\begin{abstract} We prove that the group of strict contactomorphisms of the standard tight
contact structure on the three-sphere deformation retracts to its unitary subgroup $U(2)$.
\end{abstract}
\maketitle

\vspace{-15pt}

The group $\Aut_1(\xi)$ of strict contactomorphisms of the standard tight contact structure $\xi$ on the three-sphere is known to be the total space of a fiber bundle $S^1\hookrightarrow \Aut_1(\xi) \to \SDiff(S^2)$ over the group of orientation-preserving, area-preserving diffeomorphisms of the two-sphere $S^2$,
where the projection $P$ is the map given by descending under the Hopf map $S^3 \to S^2$, and the
fiber $S^1$ is the circle subgroup of diffeomorphisms of $S^3$ which rotate all Hopf circles within
themselves by the same amount.

\begin{introthm}\label{mainthm}
In the category of Fr{\'e}chet Lie groups and $C^\infty$ maps, the fiber bundle 
$$S^1\hookrightarrow \Aut_1(\xi) \to \SDiff(S^2)$$
deformation retracts to its finite-dimensional subbundle
$$S^1\hookrightarrow U(2) \to SO(3),$$
where the $S^1$ fibers move rigidly during the deformation.
\end{introthm}

It was already known that this bundle inclusion is a homotopy equivalence, and we improve on that by showing how to lift Mu-Tao Wang's deformation retraction of the group $\SDiff(S^2)$ onto its subgroup $SO(3)$ to one of $\Aut_1(\xi)$  onto its subgroup $U(2)$.

Here are the basic definitions.

The \emph{Hopf fibration} $\HH$ of the three-sphere is a fiber bundle $S^1\hookrightarrow  S^3\xrightarrow{p} S^2$ whose fibers are the oriented unit circles on the complex lines through the origin in $\mathbb{C}^2$. The \emph{Hopf vector field} $V_{\HH}$ on $S^3$ is the unit vector field tangent to these oriented great circles. The group $\auth$  of \emph{automorphisms of} $\HH$ is the subgroup of $\Diff(S^3)$  consisting of diffeomorphisms which permute the oriented great circle fibers of $\HH$, not necessarily rigidly. They are all orientation-preserving.
The subgroup $\Aut_1(\HH)$ of \emph{strict automorphisms of $\HH$} is the subgroup of $\auth$ permuting Hopf fibers rigidly,
$\Aut_1(\HH)=\{F\in\Diff(S^3)\ \mid \  F_\dast V_{\HH} = V_{\HH} \}.$

The \emph{standard tight contact structure} $\xi$ on $S^3$ is the field of tangent two-planes which are everywhere orthogonal to the great circle fibers of the Hopf fibration. 
The \emph{standard contact one-form} $\alpha$ is the inner product with the Hopf vector field, so that $\alpha(W)=\langle V_{\HH}, W\rangle$, and therefore $\xi=\mathrm{ker}\ \alpha$.
The group $\Aut(\xi)$ is the subgroup of $\mathrm{Diff}(S^3)$  consisting of  diffeomorphisms $h$ whose differential $h_{\dast}$ permutes the tangent $2$-planes of $\xi$, meaning that $h_{\dast}$ maps the tangent 2-plane of $\xi$ at $x$ to the tangent 2-plane of $\xi$ at $h(x)$, for all $x\in S^3$. We write $h_{\dast}(\xi)=\xi$ and call $h$ a \emph{a contactomorphism}. We have that $h_{\dast}(\alpha)=\lambda \alpha$ for some smooth (always meaning $C^\infty$ here) real-valued nowhere zero function $\lambda$ on $S^3$. If $h_\dast(\alpha)=\alpha$ on the nose, then we call $h$ a \emph{strict contactomorphism} or \emph{quantomorphism}, and denote the group of these by  $\Aut_1(\xi)$.
 
 We will show in \autoref{intersection} that  the group $\Aut_1(\xi)$ consists precisely of those diffeomorphisms of $S^3$ which simultaneously preserve the Hopf fibration $\HH$ and the standard tight contact structure $\xi$, that is, $$\Aut_1(\xi)=\auth \cap \Aut(\xi).$$
This, in turn, will help us in the proof of the main theorem.

	\begin{figure}[h!]
	\begin{center}
		\includegraphics[scale=0.36]{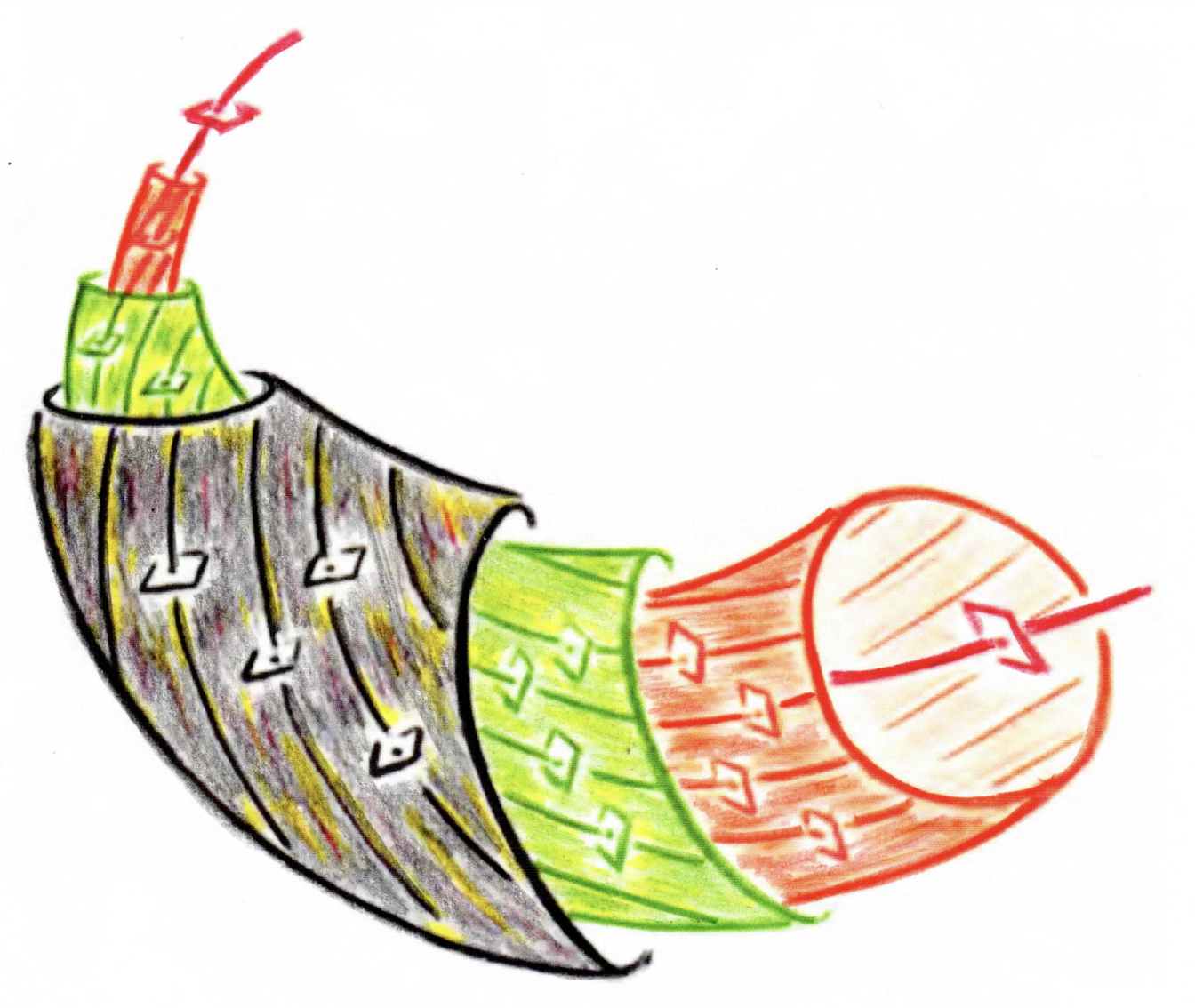}\caption{The standard tight contact structure on the three-sphere
is the field of tangent two-planes orthogonal to the great circle Hopf fibers.}\label{hopfpic}
	\end{center}
	\end{figure}

\begingroup%
\setlength{\parskip}{\storeparskip}
\setcounter{tocdepth}{2}
\tableofcontents
\endgroup%


\section{Introduction}\label{intro}

We give some historical context. The study of the homotopy types of the groups $\Diff(M)$ of diffeomorphisms of smooth manifolds $M$ and their subgroups has a rich history. By  a result of Smale, the diffeomorphism group $\Diff(S^2)$ deformation retracts to the orthogonal group $O(3)$, and by the celebrated Smale Conjecture proved by Hatcher in \cite{hatcher_smale}, the diffeomorphism group $\Diff(S^3)$ deformation retracts to the orthogonal group $O(4)$. It is natural to consider diffeomorphisms of $S^3$ which preserve extra structure there and the interplay and homotopy types of resulting moduli spaces. In this paper we are focused on the subgroup $\Aut_1(\xi)$ of  $\Diff(S^3)$ which consists of strict contactomorphisms.

The following exact sequences and bundles have been studied in the literature, and we have also presented detailed self-contained proofs in Part 2.

\noindent $(1)$ \emph{The exact sequence of Fréchet Lie algebras, equivalently, tangent spaces at the identity} $$0\to  T_\id S^1 \xrightarrow{} T_\id \Aut_1(\xi) \xrightarrow{} T_\id \SDiff(S^2)\to 0,$$

\noindent $(2)$ \emph{The exact sequence of Fréchet Lie groups}$$\{1\}\to   S^1 \xrightarrow{}  \Aut_1(\xi) \xrightarrow{}  \SDiff(S^2)\to \{1\},$$

\noindent $(3)$ \emph{The Fréchet fiber bundle}  $$S^1 \hookrightarrow  \Aut_1(\xi) \xrightarrow{}  \SDiff(S^2),$$

\noindent $(4)$ \emph{The finite-dimensional subsequence of (2) and finite-dimensional subbundle of (3)} 
$$S^1\hookrightarrow U(2) \xrightarrow{} SO(3).$$

Leslie \cite{leslie}  introduced a differential structure on the group of diffeomorphisms of a differentiable
manifold, converting it into a Fr\'echet Lie group. Banyaga \cite{banyaga1}, \cite{banyaga2} presented the above exact sequence (2) of Fr\'echet Lie groups, attributing this to Souriau \cite{souriau}, and noted its finite-dimensional exact subsequence (4).

 The fiber bundle result (3) was proved by Ratiu and Schmid in the Sobolev category \cite{ratiu_schmid}, building on work of Kostant \cite{konstant}, Souriau \cite{souriau}, Ebin and Marsden \cite{ebin_marsen}, Omori \cite{omori}, and Banyaga \cite{banyaga1}, \cite{banyaga2}. They attribute the exact sequences (1) and (2) of Fr\'echet Lie algebras and Lie groups to Kostant \cite{konstant}, and used these to derive the bundle result. 
 
 Vizman \cite{vizman} worked in the $C^\infty$ category, and obtained the exact sequences (1) and (2) above of Fr\'echet Lie algebras and Lie groups, as well as the Fr\'echet fiber bundle (3). Casals and Spacil \cite{casalsspacil} also worked in the $C^\infty$ category, attributed the Fr\'echet fiber bundle (3) above to Vizman, and showed that the inclusion of the finite-dimensional subbundle (4) into this bundle is a homotopy equivalence. Their further conclusions depended on a result of Eliashberg \cite{eliashberg92} which was only stated though not proved by him, but later proved by Eliashberg and Mishachev \cite{eliashbergmishachev}.
 
Mu-Tao Wang \cite{wang, Wang13} showed how to deformation retract the group $\SDiff(S^2)$ of orientation-preserving, area-preserving diffeomorphisms of the two-sphere to its subgroup $SO(3)$ of orthogonal
transformations by applying mean curvature flow in $S^2\times S^2$ simultaneously to the graphs of
all orientation-preserving, area-preserving diffeomorphisms of $S^2$ to itself. Our \autoref{mainthm} will be proved by lifting this to a deformation retraction of $\Aut_1(\xi)$ to $U(2)$. 

\subsection*{Organization of the paper and plan of the proof of \autoref{mainthm}}  \autoref{firstpart} begins with \autoref{prelims} where we  regard $S^3$ as the group of unit quaternions, and quickly   review left-invariant vector fields and differential forms on $S^3$. We also give a very brief overview of Fr\'echet spaces,
manifolds and Lie groups, which provide the setting for this paper.
In \autoref{intersectionsection} we will examine the behavior of diffeomorphisms which lie in the group $\Aut_1(\xi)$ of strict
contactomorphisms, and show that this group is the intersection of the groups $\auth$ and $\autxi$.
After that, here is the plan for proving the main theorem in \autoref{bundlesection}. 

We must show that the fiber bundle $S^1 \hookrightarrow  \Aut_1(\xi) \xrightarrow{}  \SDiff(S^2)$ deformation retracts to its
finite-dimensional subbundle $S^1 \hookrightarrow U(2)\rightarrow SO(3)$.
We will start with Wang's deformation retraction  \cite{wang} of the base space $\SDiff(S^2)$ to $SO(3)$, and show how to lift this to the desired deformation retraction of the total space $\Aut_1(\xi)$ to $U(2)$, in a way that moves fibers to fibers rigidly at all times, while keeping the fibers of the subbundle pointwise fixed.

To begin, we will put the standard $L^2$ Riemannian metric on $\Aut_1(\xi)$ and show that every smooth path $\gamma$ in $\SDiff(S^2)$ can be lifted to a smooth horizontal path $\overline{\gamma}$ in $\Aut_1(\xi)$, meaning one that is everywhere orthogonal to the fiber direction, and is unique once we specify its starting point.

It is natural to aim to lift Mu-Tao Wang's deformation retraction of the base space $\SDiff(S^2)$ to a deformation retraction of the total space $\Aut_1(\xi)$ by simply lifting the path followed by each
point in $\SDiff(S^2)$ to the horizontal path followed by each point in the $S^1$ fiber above it.
The problem is that although we can see that the various lifted paths $\overline{\gamma}$ in $\Aut_1(\xi)$ are smooth in the time $t$ direction, we don't yet know that they are smooth in the transverse direction.

To address this, we will start with Mu-Tao Wang's deformation in $\SDiff(S^2)$ and, using the local
product structure from the fiber bundle, define smooth ``local lifts'' of it to $\Aut_1(\xi)$, ignoring the
fact that they do not fit together coherently to a global lift. Instead, we will show how to smoothly ``adjust'' these smooth local lifts to the desired global ``horizontal-in-time'' lifts there, and so conclude that these horizontal-in-time lifts are themselves smooth.

In \autoref{secondpart},  we begin by describing nearest neighbor maps, horizontal lifts and quantitative holonomy in \autoref{appendix_nn}. In \autoref{appendix1} we compute the tangent spaces at the identity of our various Lie groups. In \autoref{lieexactsec}, \autoref{groupexactsec} and \autoref{bundlesec}, we give independent, self-contained proofs of the exactness of sequences (1) and (2) of Lie algebras and Lie groups described above, and the bundle structure of sequence (3). 
Finally, we give some background on Fréchet manifolds in \autoref{appendix_frechet}.

\subsection*{Acknowledgements} It is a pleasure to acknowledge the contributions to this project arising from conversations with Alexander Kupers and Jim Stasheff. We thank Ziqi Fang for perceptive comments on a draft of this paper. 
Merling acknowledges partial support from NSF DMS grants CAREER 1943925 and FRG 2052988. Wang acknowledges partial support from NSF GRFP 1650114.

\part{\large Deformation retraction of the strict contactomorphism group}\label{firstpart}

\section{Preliminaries}\label{prelims}
\subsection{Fr\'echet spaces and manifolds}
Fr\'echet spaces, manifolds and Lie groups provide the setting for extending the theory of finite-dimensional $C^\infty$ differentiable manifolds and $C^\infty$ maps between them to the infinite-dimensional case. We give a very brief overview here and for more details, we refer the reader to the two papers of Eells \cite{Eells58, Eells} and that of Leslie \cite{leslie} for early developments and to Hamilton's paper \cite{Ham} and the book \cite{KM} of Kriegl and Michor for a good overview with details.    We highlight in \autoref{appendix_frechet} some results which we use in the present paper.

A \emph{Fr\'echet space} $V$ is a complete metrizable vector space whose topology is induced by a
countable family of semi-norms, where a semi-norm $\rho$ behaves like a norm except that $\rho(v)=0$ does not imply that $v=0$. A simple example is the space $C^\infty[0,1]$ of $C^\infty$ maps from the interval $[0,1]$ to the real numbers. Another example is the space $\mathrm{Vect}(M)$ of $C^\infty$ vector fields on a compact $C^\infty$ manifold $M$, and yet another example is the space $S\mathrm{Vect}(M)$ of $C^\infty$ divergence-free vector fields on $M$ with respect to a Riemannian metric on $M$. The semi-norms are the usual $C^k$ norms for $k=0,1,2, \dots$ on the $C^k$ versions of these spaces. 

As in the finite-dimensional case, \emph{$C^\infty$ maps} between open subsets of Fr\'echet spaces are defined in terms of the convergence of various difference quotients. A \emph{Fr\'echet manifold} modeled on a Fr\'echet space $V$ is a Hausdorff topological space with an atlas of charts which are homeomorphisms from open sets in $V$ into $M$ such that the change of coordinate maps are $C^\infty$ maps. Basic examples are the space $\Diff(M)$ of diffeomorphisms of a compact finite-dimensional $C^\infty$ manifold $M$, equipped with the $C^\infty$ topology, which is modeled on the Fr\'echet space $\mathrm{Vect}(M)$, and the space $\SDiff(M)$ of orientation-preserving, volume-preserving diffeomorphisms of a compact Riemannian manifold, modeled on the Fr\'echet space $S\mathrm{Vect}(M)$. Both are Fr\'echet Lie groups, meaning that multiplication via composition as well as inversion are smooth maps.

In this paper, we focus on the Fr\'echet Lie group $\Aut_1(\xi)$ of strict contactomorphisms (or
quantomorphisms) of the standard tight contact structure $\xi$ on the 3-sphere $S^3$. Since $\Aut_1(\xi)$ is the total space of an $S^1$-bundle over $\SDiff(S^2)$, it is modeled, just like $\SDiff(S^2) \times S^1$, on the Fr\'echet space $S\mathrm{Vect}(S^2) \times \R$. This in turn is isomorphic to the Fr\'echet space $C^\infty(S^2)$ of $C^\infty$ real-valued functions on the two-sphere $S^2$. The Fr\'echet Lie groups $\auth$ of automorphisms of the Hopf fibration $\HH$ of $S^3$, and $\Aut(\xi)$ of contactomorphisms of the standard tight contact structure $\xi$ on $S^3$, appear briefly in this paper in the proposition that their intersection is precisely the group $\Aut_1(\xi)$, which in turn helps us to better understand $\Aut_1(\xi)$. We will study $\auth$ in a forthcoming paper.

\subsection{Left-invariant vector fields and differential forms on $S^3$}

We view the 3-sphere $S^3$ as the space of unit quaternions and make the following definitions. Let
\begin{equation}\label{eq_abc}
A(x) = x i, ~~~~\ \ B(x) = x j,~~~~ \ \ C(x) = x k,
\end{equation}
be the standard left invariant vector fields given by right multiplication by $i, j, k$. Any smooth vector field $X$ on $S^3$ can be written in the basis from \autoref{eq_abc} as
\begin{equation}\label{eq_vf_abc}
X = f A + g B + h C
\end{equation}
where $f, g$ and $h$ are smooth real-valued functions on $S^3$.  

The Lie brackets of these vector fields satisfy
\begin{equation}\label{brackets}
[A,\ B] = 2C, ~~~~ \ \ [B,\ C] = 2A, ~~~~ \ \ [C,\ A] = 2B.
\end{equation}
The dual left-invariant one-forms to $A$, $B$ and $C$ on $S^3$ with respect to the standard metric will be denoted by $\alpha, \beta$ and $\dualC$, so that
$$\alpha(A)=1,\ \ \alpha(B)=0,\ \ \alpha(C)=0,$$ and likewise for $\beta$ and $\dualC$. Their exterior derivatives are given by
\begin{equation}\label{dualforms}
d\alpha =-2\beta \wedge\dualC, ~~\ \  d\beta=-2\dualC \wedge \alpha, \ \ d\dualC=-2\alpha \wedge \beta.
\end{equation}

We choose the great circle orbits of the vector field $A$ as the fibers of our Hopf fibration $\HH$,
so that $V_{\HH} = A$. It then follows that $A$ is the Reeb vector field of $\xi$, i.e., $\alpha(A)=1$ and $d\alpha(A,-)=0$.

Viewing $A$, $B$ and $C$ as directional derivative operators,  the differential operators $\mathrm{div}$ and $\mathrm{curl}$, acting on a vector field $X$ as above, are
\begin{equation}\label{div}
\mathrm{div}(X)= Af + Bg + Ch \text{   \ \ \  and}
\end{equation}
\begin{equation}\label{curl}
 \mathrm{curl}(X)= (Bh-Cg)A + (Cf-Ah)B + (Ag-Bf)C-2X.
\end{equation}
The gradient of a smooth function $\phi : S^3 \rightarrow \mathbb{R}$ is
\begin{equation}\label{grad}
\mathrm{grad}(\phi)= (A\phi)A + (B\phi)B + (C\phi) C
\end{equation}
The formula for $\mathrm{curl}$ is derived from the identities $\mathrm{curl}(A)=-2A$, $\mathrm{curl}(B)=-2B$, and ${\mathrm{curl}(C)=-2C}$, which can be verified directly, together with the Leibniz rule $$\mathrm{curl}(\phi X)=\mathrm{grad}(\phi)\times X + \phi\mathrm{curl}(X).$$

\section{The group  $\Aut_1(\xi)$ of strict contactomorphisms}\label{intersectionsection}
In this section we characterize the strict contactomorphisms of the standard tight contact structure $\xi$. 
\begin{prop}\label{intersection}
The group of strict contactomorphisms is the intersection of the contactomorphism group with the automorphism group of the Hopf fibration, 
$$\Aut_1(\xi)=\auth\cap \Aut(\xi).$$
\end{prop}

\begin{proof}
We begin by showing that $\Aut_1(\xi)\subseteq \auth\cap \Aut(\xi)$. Suppose $F\in \Aut_1(\xi)$, so by definition $F$ is a diffeomorphism of $S^3$ that satisfies $F^\dast  \alpha=\alpha$. Recall that the  Reeb vector field $A$ associated with the 1-form $\alpha$   is uniquely characterized by
$$\alpha(A)=1 \text{ \ \  \  and \ \  \  } d\alpha(A,-)=0.$$
We consider the pushforward $F_\dast A$ of the vector field $A$ by the diffeomorphism $F$, and note that 
$$\alpha(F_\dast A)=(F^\dast \alpha)(A)=\alpha(A)=1,$$ and
$$d\alpha(F_\dast A, -)= F^\dast(d\alpha)(A, -)=d(F^\dast \alpha)(A,-)=d\alpha(A, -)=0.$$
By uniqueness of Reeb vector fields, we  have $F_\dast A= A$, so $F\in  \auth$. Therefore $F$ is in the intersection $\auth\cap \Aut(\xi)$.

Next, suppose that $F\in \auth\cap \Aut(\xi)$, so that 
$$F_\dast A=\lambda A \text{\ \ \  and \ \ \ } F^\dast \alpha =\mu \alpha,$$ where $\lambda$ and $\mu$ are smooth real-valued, positive functions on $S^3$. This gives us that
$$(F^\dast \alpha)(A)=\alpha(F_\dast A)=\alpha(\lambda A)=\lambda \alpha(A)=\lambda,$$ while at the same time
$$ (F^\dast \alpha)(A)=(\mu \alpha)(A)=\mu(\alpha(A))=\mu,$$ so it follows that $\lambda=\mu$. We now show that $\lambda=1$, so that $F^\dast \alpha=\alpha$, which will imply $F\in \Aut_1(\xi)$.  

Recall that $A,B,C$ is the left-invariant orthonormal frame field on $S^3$. Note that
$$(d\alpha) (F_\dast A, F_\dast B)=(d\alpha)(\lambda A, F_\dast B)=\lambda (d\alpha)(A, F_\dast B)=0,$$ while at the same time
\begin{eqnarray*}(d\alpha) (F_\dast A, F_\dast B)&=&F^\dast (d\alpha)(A,B)=d(F^\dast \alpha)(A,B)=d(\mu  \alpha)(A,B)=d(\lambda \alpha)(A,B)\\
&=& (d\lambda \wedge \alpha +\lambda d\alpha)(A,B)=(d\lambda \wedge \alpha)(A,B)\\
&=& (d\lambda)(A)\alpha(B)-(d\lambda)(B)\alpha(A)=-(d\lambda)(B)=-B(\lambda).
\end{eqnarray*}
Thus $B(\lambda)=0$. Similarly, we can show $C(\lambda)=0$. Lastly,
$$ \textstyle A(\lambda)=\frac{1}{2}[B,C]\lambda=\frac{1}{2}(BC-CB)\lambda=0,$$ and
 hence the function $\lambda\colon S^3\to \R$ must be constant. But since Hopf fibers are taken to Hopf fibers with $F_\dast A=\lambda A$ for constant $\lambda$, then $\lambda$ must be identically 1. Thus $\mu=1$, and so $F\in \Aut_1(\xi)$.
\end{proof}

We collect a few more useful properties of strict contactomorphisms. Note that the following proposition falls out of the proof of \autoref{intersection}, where we showed that for diffeomorphisms $F\in \Aut_1(\xi)=\auth\cap \Aut(\xi)$, we must have $F_\dast A=A$, namely they permute Hopf fibers rigidly.

\begin{prop}\label{rigidHopf}
The strict contactomorphism group $\Aut_1(\xi)$ is a subgroup of the strict automorphism group $\Aut_1(\HH)$ of the Hopf fibration.
\end{prop}

Lastly, we record how elements in the simultaneous automorphism group of the Hopf fibration and the standard tight contact structure behave with respect to volume on $S^3$ and area on $S^2$.

\begin{prop}\label{areapres}
The diffeomorphisms of $S^3$ in $\Aut_1(\xi)$ are volume-preserving on $S^3$ and project to area-preserving diffeomorphisms of $S^2$ under the Hopf projection map $p$.
\end{prop}

\begin{proof}
Let $F\in \Aut_1(\xi)$, so that  $F^\dast\alpha=\alpha$. Hence $F$  takes the contact tangent 2-plane distribution $\xi$ to itself. We show that $F_\dast$ takes these tangent 2-planes to one another in an area-preserving way, as follows.

Recall the formulas that the dual forms to $A,B,C$ satisfy from \autoref{dualforms} and note that the area form on the tangent 2-planes in the distribution $\xi$ is $\beta\wedge \dualC$. We compute
\begin{eqnarray*}
\textstyle (\beta\wedge \dualC)(F_\dast B, F_\dast C)&=& \textstyle -\frac{1}{2}d\alpha (F_\dast B, F_\dast C)= -\frac{1}{2} (F^\dast d\alpha )(B,C)\\
\textstyle &=& \textstyle -\frac{1}{2} d(F^\dast\alpha )(B,C)=-\frac{1}{2}d\alpha (B,C)\\
\textstyle &=& (\beta\wedge \dualC)(B,C)=1.
\end{eqnarray*}
Thus indeed $F_\dast$ takes the 2-planes in the distribution $\xi$ to one another in an area preserving way.

We finish as follows. By \autoref{intersection}, $F_\dast(A)=A$, telling us that $F$ permutes Hopf fibers rigidly. And as we just saw above, $F_\dast$ is area-preserving on the tangent 2-planes orthogonal to the Hopf fibers. So it follows that $F$ is volume-preserving on $S^3$. Finally, since the Hopf projection $p\colon S^3\to S^2$ is up to scale a Riemannian submersion (it doubles lengths in $S^3$ orthogonal to the Hopf fibers), it follows that the diffeomorphism $F$ of $S^3$ projects to an area-preserving diffeomorphism of $S^2$, as claimed.
\end{proof}

\section{Bundle deformation retraction}\label{bundlesection}

Since we have shown that $\Aut_1(\xi) =\auth \cap \Aut(\xi)$, we know that each diffeomorphism $F$
of $S^3$ which lies in $\Aut_1(\xi)$ also lies in $\auth$, which means that it permutes the fibers of the
Hopf fibration $\HH$ and hence induces a diffeomorphism $f$ of $S^2$. The projection map $P\colon \Aut_1(\xi) \to \SDiff(S^2)$ is then defined by $P(F) = f$. We know from \cite{vizman} that this is a bundle with fiber $S^1$. We now begin the proof of our main theorem, namely that  the bundle
\begin{equation}
S^1 \hookrightarrow \mathrm{Aut}_1(\xi) \xrightarrow{P} \mathrm{SDiff}^+(S^2)
\end{equation}
deformation retracts to its finite-dimensional subbundle
\begin{equation}
S^1 \hookrightarrow U(2)  \rightarrow SO(3).
\end{equation}

We will prove this by starting with Wang's deformation retraction \cite{wang}  of the bigger base space $\mathrm{SDiff}^+(S^2)$ to the smaller base space $SO(3)$ and lifting it to the desired deformation retraction of the bigger total space $\mathrm{Aut}_1(\xi)$ to the smaller one, $U(2)$, in a way that moves fibers rigidly  throughout the deformation retraction while keeping fibers of the subbundle fixed pointwise.

\subsection{The standard $L^2$ Riemannian metric on $\Aut_1(\xi)$}	
To facilitate the lifting, we equip $\mathrm{Aut}_1(\xi)$ with the  $L^2$ Riemannian metric. Let $X$ and $Y$ be $C^\infty$ vector fields on $S^3$. We define their inner product as \begin{gather}\label{eq_l2_metric_id}
\begin{split}
\langle X,\, Y \rangle_{L^2} = \frac{1}{2\pi^2} \int\limits_{S^3} \langle X(x), Y(x) \rangle\, d\mathrm{vol}_x,
\end{split}
\end{gather}
where the point $x$ ranges over $S^3$ and where $d\mathrm{vol}_x$ is the Euclidean volume element on $S^3$.  The scale factor $ \frac{1}{2\pi^2} $ lets unit vector fields on $S^3$ have $L^2$ length equal to 1, since the volume of $S^3$ is $2\pi^2$, and this will simplify expressions later on.

The left-invariant vector field $A$ on $S^3$, given by $A(x)=xi$, lies in $T_{\mathrm{id}}\mathrm{Aut}_1(\xi)$, and its $L^2$ length is 1. By contrast, the left-invariant vector fields $B$ and $C$ on $S^3$ do \emph{not} lie in $T_{\mathrm{id}}\mathrm{Aut}_1(\xi)$.

At other points  $F \in \mathrm{Aut}_1(\xi)$, an element of the tangent space $T_F\Diff(S^3)$ is a vector field in $S^3$ along the diffeomorphism $F$, meaning  that it assigns to each point $x\in S^3$ a tangent vector to $S^3$ at the point $F(x)$. We will denote such an element of $T_F\Diff(S^3)$ by the symbol $X\circ F$, where $X$ is a smooth vector field on $S^3$ and where $X\circ F$ assigns to each point $x\in S^3$ the tangent vector $X(F(x))\in T_{F(x)}S^3$.

Then our $L^2$ Riemannian metric at the point  $F \in \mathrm{Aut}_1(\xi)$ is given by
\begin{eqnarray*}
\langle X \circ F ,\, Y \circ F \rangle_{L^2} &=&  \frac{1}{2\pi^2}  \int\limits_{S^3} \langle (X\circ F) (x), (Y \circ F) (x) \rangle\, d\mathrm{vol}_{x}\\
&=&  \frac{1}{2\pi^2}  \int\limits_{S^3} \langle X(x), Y(x) \rangle\, d\mathrm{vol}_{x},
\end{eqnarray*}
and is well-defined because all the diffeomorphisms $F$ of $S^3$ which lie in $\Aut_1(\xi)$ are volume-preserving by \autoref{areapres}.

This is a smooth, \emph{weak} Riemannian metric on $\mathrm{Aut}_1(\xi)$ in the sense that the topology induced on $\mathrm{Aut}_1(\xi)$ by the $L^2$ norm has fewer open sets than the $C^\infty$ topology.

The $L^2$ Riemannian metric on $\Aut_1(\xi)$ is right-invariant, but not left-invariant. The $S^1$ subgroup of $\Aut_1(\xi)$ which rotates all Hopf fibers by the same amount consists of isometries in this metric, and is the center of the group  $\Aut_1(\xi)$. The Fr\'echet group  $\Aut_1(\xi)$ is a Fr\'echet Lie subgroup of the Fr\'echet Lie group $\SDiff(S^3)$ of volume-preserving and orientation-preserving diffeomorphisms of $S^3$. 

\subsection{Lifting a single curve in S$\Diff^+(S^2)$ to $\Aut_1(\xi)$}
A path in $\Aut_1(\xi)$ will be said to be \emph{horizontal} if it is everywhere orthogonal to the $S^1$ fiber direction with respect to the $L^2$ Riemannian metric. Lifting paths in $\mathrm{SDiff}^+(S^2)$ to horizontal paths in $\mathrm{Aut}_1(\xi)$ will play a key role in the proof of our main theorem, so we establish the following lemma first.

\begin{lem}[Lifting Lemma]\label{lem_lift}
	Let $\gamma\colon [0, 1] \rightarrow \mathrm{SDiff}^+(S^2)$ be a smooth path in $\mathrm{SDiff}^+(S^2)$ with $\gamma(0)=f_0$, and let $F_0$ be an element of $\mathrm{Aut}_1(\xi)$ such that $P(F_0) = f_0$. Then there exists a unique horizontal path  $\overline{\gamma}\colon [0, 1] \rightarrow \mathrm{Aut}_1(\xi)$ such that $\overline{\gamma}(0) = F_0$ and $P(\overline{\gamma}) = \gamma$. 
\end{lem}

\vspace{-20pt}
\begin{figure}[h!]
	\begin{center}
		\includegraphics[scale=0.36]{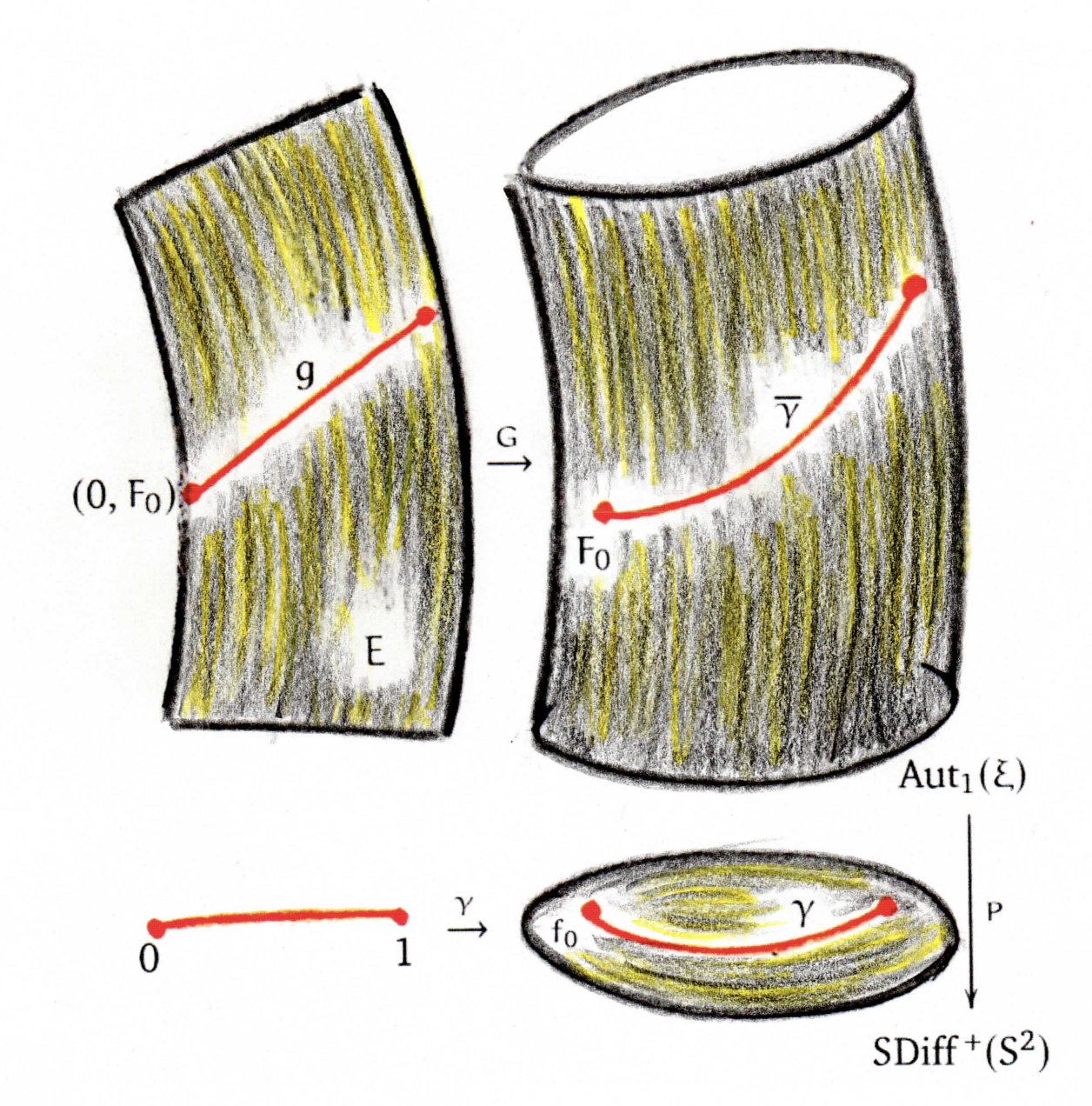}\caption{The path $\gamma$ in $\SDiff(S^2)$ lifts to the horizontal path $\overline{\gamma}$ in $\Aut_1(\xi)$}\label{pathliftspic}
	\end{center}
	\end{figure}

\vspace{-15pt}	
\begin{proof}
We start with the \emph{quantomorphism bundle}
\begin{equation}
Q \colon S^1 \hookrightarrow \mathrm{Aut}_1(\xi) \xrightarrow{P} \mathrm{SDiff}^+(S^2)
\end{equation}
and then use the map $\gamma\colon [0, 1] \rightarrow \mathrm{SDiff}^+(S^2)$ to construct the pullback bundle
\begin{equation}
\gamma^\dast Q \colon S^1 \hookrightarrow E \rightarrow [0, 1]
\end{equation}
over the interval $[0, 1]$. The familiar pullback construction extends to the category of Fr\'echet
manifolds and smooth maps \cite{KM}. The points of the total space $E$ are, as usual, the pairs $(t, F)$,
where $t \in [0, 1]$ and $F \in P^{-1}(\gamma(t))$. Since the base space is an interval, the total space $E$ is trivial, that is, an annulus
diffeomorphic to the product $[0, 1] \times S^1$. The bundle map $G\colon \gamma^\dast Q \rightarrow Q$ is defined by 
$G(t, F) = F$.


The smooth $L^2$ Riemannian metric on the Fr\'echet manifold $\mathrm{Aut}_1(\xi)$  pulls back to a smooth Riemannian metric on the annulus $E$. The horizontal tangent hyperplane distribution on $\mathrm{Aut}_1(\xi)$ is by definition the $L^2$ orthogonal
complement to the one-dimensional vertical fiber direction there. It pulls back to a smooth tangent line field on the annulus $E$ which is transverse to the vertical fiber direction there. Though it may not look horizontal to Euclidean eyes, we will say that this line field is ``horizontal'' on $E$.

Since $E$ is finite-dimensional, by the usual existence and uniqueness theorems for ordinary differential equations we get a horizontal path $t \mapsto g(t)$ on $E$ which begins at the point $(0, F_0)$. In particular, it is a cross-section of the pullback bundle $\gamma^\dast Q$. 

Pushing this horizontal path $g$ in $E$ forward by the bundle map $G\colon \gamma^\dast Q \rightarrow Q$, we get the
desired lift $\overline{\gamma}(t) = G(g(t))$ of $\gamma$ to a horizontal path in $\mathrm{Aut}_1(\xi)$ which begins at the given point $F_0$ in the fiber $P^{-1}(f_0)$. 

This completes the proof of the lifting lemma for single curves.
\end{proof}

\begin{rem}
If we let $F_0$ vary over all the points in the $S^1$-fiber $P^{-1}(\gamma(0))$, we get a circle's worth of disjoint lifts of $\gamma$ which are carried to one another by the action of the subgroup $S^1$ of $\Aut_1(\xi)$.
\end{rem}

\subsection{Lifting families of curves in S$\Diff^+(S^2)$ to $\Aut_1(\xi)$} Let $$\Phi\colon \mathrm{SDiff}^+(S^2) \times [0, 1] \rightarrow \mathrm{SDiff}^+(S^2)$$ be any smooth deformation of $\SDiff(S^2)$ within itself, meaning that $\Phi(f,0)=f$, without any other requirements. Then for each $f\in \SDiff(S^2)$ we have a smooth path $\gamma(t)=\Phi(f,t)$ in $\SDiff(S^2)$, and these paths vary smoothly with the choice of initial point $f$. By the Lifting \autoref{lem_lift}, we can lift each of these paths uniquely  to a horizontal path $\overline{\gamma}$ in $\Aut_1(\xi)$ once we specify its initial point $F\in P^{-1}(f).$ 

We know that each lifted path is smooth in the time parameter $t$, but we do not yet know that the collection of  lifts is smooth in the ``transverse direction'', meaning smoothly dependent on the initial points $F\in \Aut_1(\xi)$. In this section we prove smooth dependence on initial points.

As mentioned earlier, our plan is to define smooth ``local lifts'' of these paths, ignoring the fact that they do not fit together coherently to a global lift, and then show how to smoothly ``adjust'' these to the desired ``horizontal-in-time'' lifts, which are defined globally, and so conclude that they are indeed smoothly dependent on their initial points.

We start with the following lemma.

\begin{lem}\label{pieceslem}
For each point $f$ in $\SDiff(S^2)$, there is an open neighborhood $U$ of $f$ and a partition $0=t_0<t_1<\dots<t_n=1$ of the interval $[0,1]$ such that each image $\Phi(U\times [t_{k-1}, t_k])$ lies in an open set in $\SDiff(S^2)$ over which the $\Aut_1(\xi)$ bundle is trivial.
\end{lem}

\begin{figure}[h!]
	\begin{center}
		\includegraphics[scale=.25]{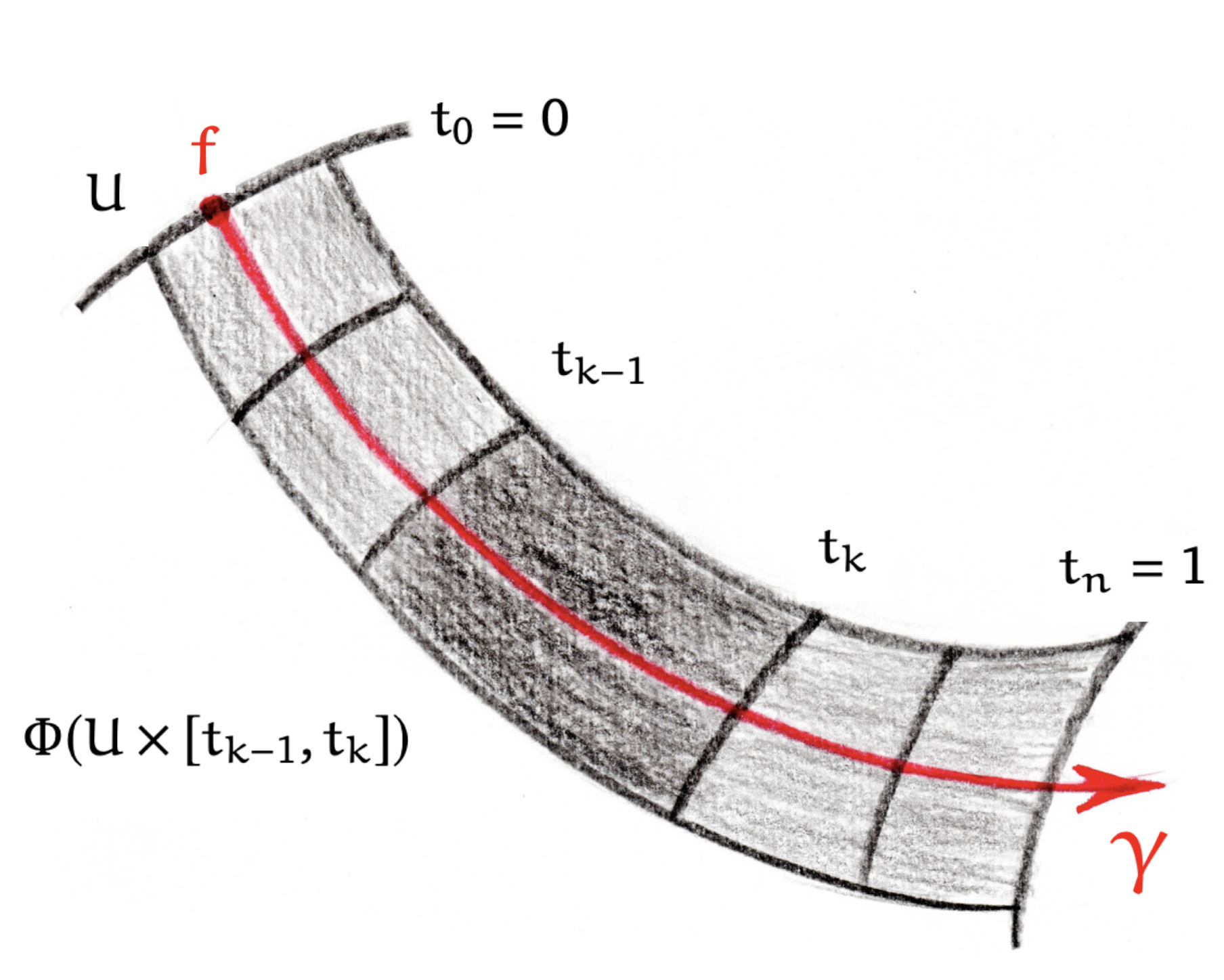}\caption{The $\Aut_1(\xi)$ bundle is trivial over each piece.} \label{}
	\end{center}
	\end{figure}
	
\begin{proof}
It follows from the continuity of our deformation $\Phi$ that for each point $(f,t)$ in its domain, there is an open neighborhood $U_t$ of $f$ in $\SDiff(S^2)$ and a real number $\epsilon_t$ such that the image $\Phi(U_t\times (t-\epsilon_t, t+\epsilon_t))$ lies in an open set in $\SDiff(S^2)$ over which our $\Aut_1(\xi)$-bundle is trivial. 

By compactness finitely many of these open intervals $(t-\epsilon_t, t+\epsilon_t)$ cover $[0,1]$, and we can simply let $U$ be the intersection of the finitely many corresponding open sets $U_t$, and choose a partition of $[0,1]$ subordinate to this covering of $[0,1]$. This proves the lemma.
\end{proof}

\subsubsection*{Defining the smooth local lifts} We choose any point $f\in \SDiff(S^2)$ and focus on one of the pieces $\Phi(U\times [t_{k-1}, t_k])$ of our tubular neighborhood $\Phi(U\times [0,1])$ of the curve $\gamma=\Phi(f\times [0,1])$. By \autoref{pieceslem}, there is an open set $V\subset \SDiff(S^2)$ which contains this piece and over which the bundle $$S^1\hookrightarrow \Aut_1(\xi)\xrightarrow{P} \SDiff(S^2)$$ is trivial. Let $\sigma\colon V\times S^1\to \Aut_1(\xi)$ be a smooth trivialization of this bundle over $V$. Then, picking and fixing any point $\varphi\in S^1$, we have a smooth local lift 
$$\sigma(\Phi(f,t), \varphi), \text{ where } \ f\in U,\ t\in[t_{k-1},t_k]$$ of the $k^{\text{th}}$ piece of $\Phi(U\times [0,1])$ to the total space $\Aut_1(\xi)$ of our bundle, as desired. 

\subsubsection*{Adjusting the local lifts to prove that the global horizontal lift is smooth}

To simplify the notation, let $F\in P^{-1}(f)$, and define 
$$F_t=\sigma(\Phi(f,t),\varphi)\in \Aut_1(\xi)\subset \Diff(S^3).$$
We want to adjust each such diffeomorphism $F_t$ along the $S^1$ fiber through it in $\Aut_1(\xi)$ by an angle $\theta(f,t)$ so that the corrected family of diffeomorphisms $$G_t=F_t\ e^{i\theta(f,t)}$$ is horizontal with respect to $t$ in the $L^2$ Riemannian metric on $\Aut_1(\xi)$ for each $f\in U$. In this notation, $e^{i\theta(f,t)}$ denotes the diffeomorphism of $\Aut_1(\xi)$ which rotates all $S^1$ fibers through the angle $\theta(f,t)$. For simplicity of notation we will write $\theta(t)$ instead of $\theta(f,t)$, and tacitly understand dependence of this angle on the initial diffeomorphism $f$ of $S^2$.

	
Regard the right side of the equation $$G_t=F_t \ e^{i\theta(t)}$$ as a product in the group $\Aut_1(\xi)$, and apply the Leibniz Rule when differentiating it with respect to time $t$ to get
\begin{eqnarray*}
\frac{d}{dt} G_t &=& \frac{d}{dt}(F_t \ e^{i\theta(t)})\\
&=& \left(\frac{d}{dt} F_t\right)e^{i\theta(t)}+F_t\left( \frac{d}{dt} e^{i\theta(t)}\right)\\
&=& (X_t\circ F_t)\ e^{i\theta(t)}+F_t(e^{i\theta(t)}i\theta'(t))\\
&=& X_t\circ G_t+G_t \ i\theta'(t),
\end{eqnarray*}
where $X_t$ is the time-dependent vector field on $S^3$ generated by the one-parameter family of diffeomorphisms $F_t$ of $S^3$, so that $\frac{d}{dt} F_t= X_t\circ F_t$, and where $X_t\circ G_t$ is a vector field along $G_t$.

In the last equation, the first term $X_t\circ G_t$ is a vector field on $S^3$, and so is the second term $G_t i\theta'(t)$, even though it may not look so at first glance. The vector field $G_t i\theta'(t),$ when evaluated at a point $x\in S^3$, lies in $T_{G_t(x)}S^3$,  is tangent to the Hopf fiber through that point, and is scaled to have length $\theta'(t)$. That is the same as the vector field $A$ at the point $G_t(x)$, scaled to length $\theta'(t)$. So we can write
$$G_t(x)\ i\theta'(t)=A(G_t(x))\ \theta'(t),$$
or dropping the point $x$ from the notation, we have 
$$G_t\ i\theta'(t)=(A\circ G_t)\ \theta'(t).$$
Inserting this into the last term of our above computation of the derivative $\frac{d}{dt} G_t$, we get $$\frac{d}{dt} G_t=X_t\circ G_t + (A\circ G_t)\ \theta'(t),$$ and will continue on from here. 

We keep in mind that our goal is to find the family of rotations $e^{i\theta(t)}$ of $S^3$ which will make the ``adjusted'' curves of quantomorphisms
$$G_t(x)=F_t(x)\ e^{i\theta(t)}$$ horizontal in time with respect to the $L^2$ Riemannian metric on $\Aut_1(\xi)$. To this end, we consider the tangent space to $\Aut_1(\xi)$ at any point $G$, and let $\pi$ denote its projection to the one-dimensional ``vertical'' subspace tangent to the $S^1$-fiber direction,
$$\pi\colon T_G\Aut_1(\xi)\to \mathrm{Vert}_G.$$

Then we write
\begin{eqnarray*}
\pi(X\circ G)&=&\langle X\circ G, A\circ G\rangle_{L^2}\ (A\circ G)\\
&=& \langle X, A\rangle_{L^2}\ (A\circ G),
\end{eqnarray*}
thanks to the invariance of our $L^2$ metric under the action of $G$, and to the fact that $A\circ G$ is a unit vector tangent to $\Aut_1(\xi)$ at $G$.

Now we apply this vertical projection $\pi$ to our earlier equation, and set the result equal to zero to require it to be horizontal,
\begin{eqnarray*} 0=\pi\left(\frac{d}{dt} G_t\right) &=& \pi\Bigl(X_t\circ G_t+(A\circ G_t)\ \theta'(t)\Bigr)\\
&=&\langle X_t, A\rangle_{L^2}\ (A\circ G_t) +\langle A, A \rangle_{L^2}\ \theta'(t)\ (A\circ G_t).
\end{eqnarray*}

We drop the vertical vector $(A\circ G_t)$ from above and save only its coefficient, recall that $\langle A, A\rangle_{L^2}=1$, and  are left with the scalar equation
$$0=\langle X_t, A\rangle_{L^2} +\theta'(t),$$
or equivalently
$$\theta'(t)=- \langle X_t, A\rangle_{L^2}.$$

We recall that $\frac{d}{dt} F_t=X_t\circ F_t$, so $X_t=\frac{d}{dt} F_t\circ F_t^{-1}.$ Inserting this above, we get
$$\langle X_t, A\rangle_{L^2}= \left\langle \frac{d}{dt} F_t\circ F_t^{-1}, A\right\rangle_{L^2} =  \left\langle \frac{d}{dt} F_t, A\circ F_t \right\rangle_{L^2}.$$
Thus
$$\theta'(t)= -\left\langle \frac{d}{dt} F_t, A\circ F_t \right\rangle_{L^2}.$$

This makes sense because we want to eliminate the vertical component of $\frac{d}{dt}F_t$ in order to move $F_t$ to $G_t$. Integrating, we get

\begin{eqnarray*}
\theta(t) &=& \theta(t_{k-1})-\int_{t_{k-1}}^t \left\langle  \frac{d}{ds} F_s, A\circ F_s\right\rangle_{L^2} \ ds\\
&=& \theta(t_{k-1}) -\frac{1}{2\pi^2} \int_{t_{k-1}}^t \int_{S^3} \left\langle  \frac{d}{ds} F_s(x), A\circ F_s(x)\right\rangle \ d\mathrm{vol}_x\  ds.
\end{eqnarray*}

The last equation tells us that the adjusting angle $\theta(t)$ depends smoothly on the initial angle $\theta(t_{k-1})$ and on $F_t=\sigma(\Phi(f,t), \varphi)$, which itself depends smoothly on the diffeomorphism $f\in \SDiff(S^2)$ and  the time $t$.

On each subinterval $[t_{k-1}, t_k]$, the initial angle $\theta(t_{k-1})=\theta(f, t_{k-1})$ depends smoothly on $f$ by the above construction for the preceding time interval $[t_{k-2}, t_{k-1}]$, and we start with $\theta(0)=0$.

Thus the adjusted family of quantomorphisms $$G_t(x)=F_t(x)\ e^{i\theta(t)},$$ with $F_t=\sigma(\Phi(f,t), \varphi)$, $f\in U$ and $t\in [t_{k-1}, t_k]$, is horizontal in time and depends smoothly on $(f,t)$, which is exactly what we were aiming for.

\subsection{Completing the proof of the main theorem} We adjust notation as follows. Let $\gamma(t)$ be a smooth path in $\SDiff(S^2)$, with $t\in [0,1]$, beginning at the point $\gamma(0)=f_0$, and let $F_0$ be a point in $P^{-1}(f_0)$. Then we will write 
$$\overline{\gamma}(t)=\overline{\gamma}(\gamma, F_0, t)$$ to designate the horizontal path in $\Aut_1(\xi)$ which covers $\gamma$ and which begins at the point $\overline{\gamma}(0)= F_0$. This is the path of quantomorphisms that we called $G_t$ above. 

Now we  have all the ingredients we need to complete the proof of our main theorem, which we restate.

\begin{thm}\label{U2def}
	In the category of Fr{\'e}chet Lie groups and $C^\infty$ maps, the fiber bundle 
$$S^1\hookrightarrow \Aut_1(\xi) \to \SDiff(S^2)$$
 deformation retracts to its finite-dimensional subbundle
$$S^1\hookrightarrow U(2) \to SO(3),$$
where the $S^1$ fibers move rigidly during the deformation. \end{thm}

	\begin{figure}[h!]
	\begin{center}
		\includegraphics[scale=0.45]{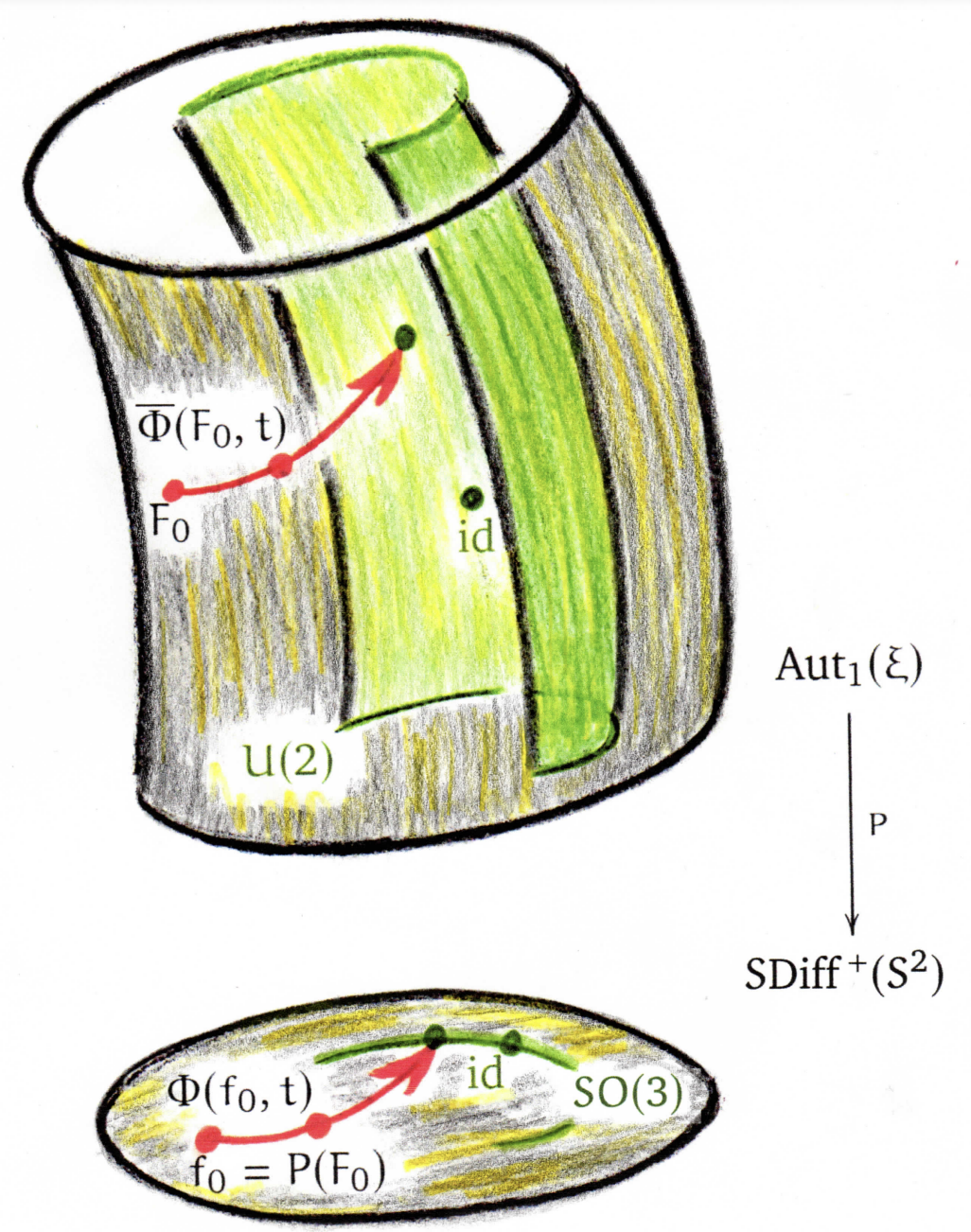}\caption{Deformation retraction of $\Aut_1(\xi)$ onto $U(2)$}\label{deformationpic}
	\end{center}
	\end{figure}	

\begin{proof}
	Let $\Phi\colon \SDiff(S^2) \times [0, 1] \rightarrow \SDiff(S^2)$ be the deformation retraction of $\SDiff(S^2)$ to the orthogonal group $SO(3)$ given by Mu-Tao Wang's theorem from \cite{wang}. We lift $\Phi$ to a deformation retraction
	\begin{equation}
	\overline{\Phi}\colon \mathrm{Aut}_1(\xi) \times [0, 1] \rightarrow \mathrm{Aut}_1(\xi)
	\end{equation}
	of $\mathrm{Aut}_1(\xi)$ to its subgroup $U(2)$ by defining
	\begin{equation}
	\overline{\Phi}(F_0,\, t) = \overline{\gamma}\big(\gamma,\ F_0, \ t\big),
	\end{equation}
	where $\gamma$ is the path in $\SDiff(S^2)$ which starts at the point $f_0 = P(F_0)$ and follows Wang's deformation retraction, $\gamma(t)=\Phi(f_0, t)$, and where $\overline{\gamma}(\gamma, F_0, t)$ is the horizontal lift of $\gamma$ defined above.
	
	The deformation retraction $\overline{\Phi}$ of $\Aut_1(\xi)$ moves along horizontal curves which cover the corresponding paths of the deformation retraction $\Phi$ of $\SDiff(S^2)$. The subgroup $S^1$ of $\Aut_1(\xi)$ which rotates all Hopf fibers by the same amount consists of isometries in this metric, and so carries horizontal paths to horizontal paths. Thus the $S^1$ fibers of $\Aut_1(\xi)$ move rigidly among themselves during the deformation retraction $\overline{\Phi}.$ 
	
	At the end of the deformation retraction, $\Phi$ has compressed $\SDiff(S^2)\times \{1\}$ to the orthogonal group $SO(3)$, and $\overline{\Phi}$ has compressed $\mathrm{Aut}_1(\xi) \times \{1\}$ to the unitary group $U(2)$.
	
	A point $f_0$ in $\SDiff(S^2)$ which starts out in the subgroup $SO(3)$ does not move during this process, and likewise a point $F_0$ in $\mathrm{Aut}_1(\xi)$ which starts out in the subgroup $U(2)$ does not move.
	This completes the proof of our main theorem.
\end{proof}

\newpage
\part{\large The Fr\'echet bundle structure of the space of strict contactomorphisms}\label{secondpart}

We begin this part of our paper by describing nearest neighbor maps, horizontal lifts and
quantitative holonomy, and then compute the tangent spaces at the identity of our various Lie
groups. After that, we give independent, self-contained proofs of the exactness of our
sequence of Lie algebras and the exactness and bundle structure of our sequence of Lie groups,
proved earlier by the many mathematicians cited in \autoref{intro}.

\section{Nearest neighbor maps, horizontal lifts and quantitative holonomy}\label{appendix_nn}


\subsection{Nearest neighbor maps} 
Let $C$ and $C'$ be two Hopf fibers on $S^3$ which are not orthogonal to one another, or equivalently, whose projections to $S^2$ are not antipodal. These two Hopf fibers are a constant distance, say $\delta < \pi/2$ apart on $S^3$.

Thus, each point $x$ on $C$ has a unique \emph{nearest neighbor} $x'$ on $C'$, which is the point that minimizes the distance between $x$ and $C'$. Similarly, $x'$ on $C'$ has $x$ on $C$ as its nearest neighbor there. Furthermore, the correspondence between $x$ on $C$ and $x'$ on $C'$ is an isometry between these two circles. 

The nearest neighbor map between the Hopf fibers $C'$ and $C$ takes the point
$$x' = (\cos\delta \cos\theta, \cos\delta \sin\theta, \sin\delta \cos\phi, \sin\delta \sin\phi) \text{  on  } C'$$ to the point $x= (\cos\theta, \sin\theta, 0, 0)$ on  $C$, as depicted in \autoref{neighborpic}.

\begin{figure}[h!]
\begin{center}
	\includegraphics[scale=0.4]{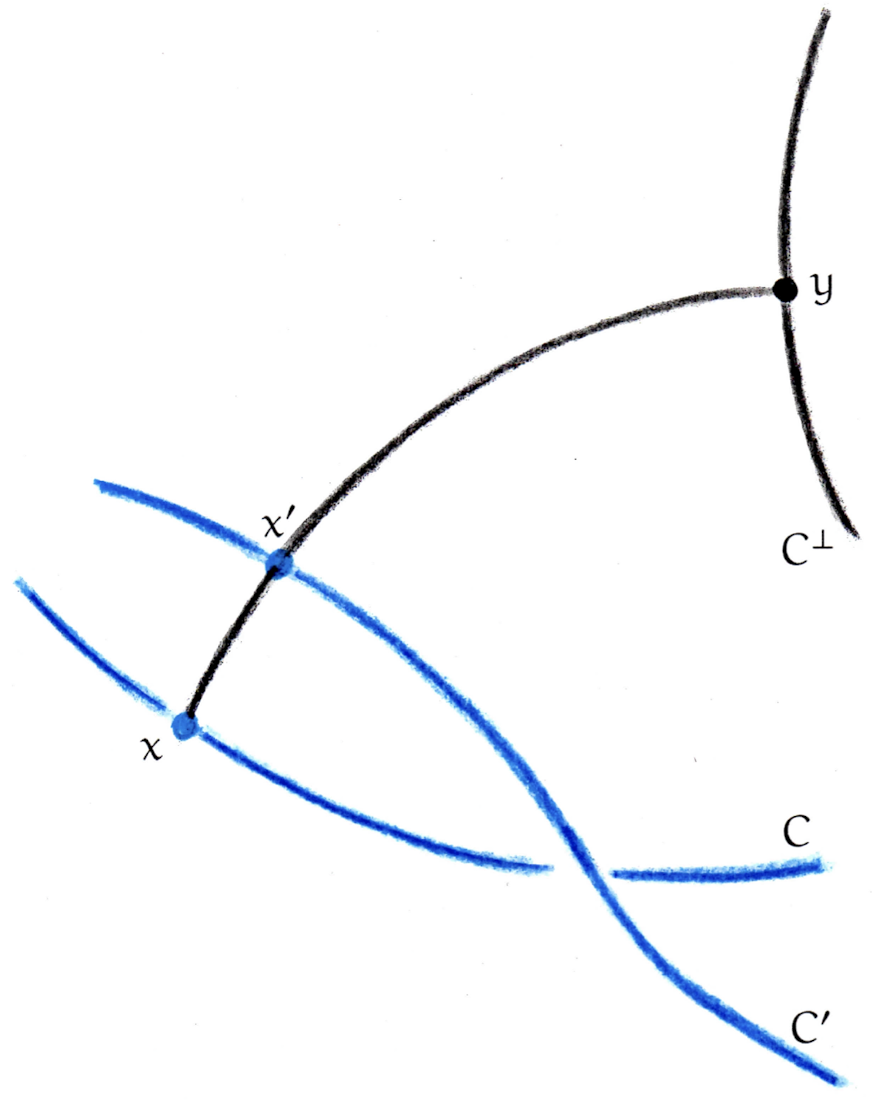}\caption{The points $x$ and $x'$ are nearest neighbors on the great circles $C$ and $C'$}\label{neighborpic}
\end{center}
\end{figure}


The composition of nearest neighbor maps $C \rightarrow C' \rightarrow C''$ is \emph{not} necessarily the nearest neighbor map $C \rightarrow C''$, and if we move along a succession of nearest neighbor maps out from $C$ and eventually back again to $C$, the composition will be some rotation of $C$. In related settings, a similar phenomenon is called \emph{holonomy}, so we will use that term here as well.

\subsection{Horizontal lifts}\label{horizontal} Consider the Hopf projection $p \colon S^3 \rightarrow S^2$ and let $\gamma\colon [0, 1] \rightarrow S^2$
be a smooth curve. Given a point $x$ on the Hopf fiber $p^{-1}(y)$, there exists a  smooth curve ${\overline{\gamma}\colon [0, 1] \rightarrow S^3}$ which is unique and runs always orthogonal to Hopf fibers, covers $\gamma$ in the sense that $p \circ \overline{\gamma} = \gamma$ and satisfies $\overline{\gamma}(0) = x$. We refer to $\overline{\gamma}$ as a \emph{horizontal lift} of $\gamma$ because we think of Hopf fibers as being ``vertical'' and the orthogonal tangent 2-planes as being ``horizontal''. In fact, viewing $S^3$ as a principal $U(1)$-bundle over $S^2$, the horizontal lift is parallel transport with respect to the connection defined by the 1-form $\alpha$. If $\gamma$ is a geodesic in $S^2$ between
the non-antipodal points $y_1$ and $y_2$, then the horizontal lifts of $\gamma$ give us the nearest
neighbor map between the Hopf fibers $p^{-1}(y_1)$ and $p^{-1}(y_2)$.

\subsection{Quantitative holonomy}\label{quantitativeholonomy}

In the Hopf fibration $\HH$, we choose radius $1/2$ for the base 2-sphere, so that the projection map $p\colon S^3\to S^2(\frac{1}{2})$ is a Riemannian submersion, meaning that its differential takes tangent 2-planes
orthogonal to the Hopf fibers isometrically to their images in $S^2(\frac{1}{2})$.

\begin{figure}[h!]
\begin{center}
	\includegraphics[scale=0.45]{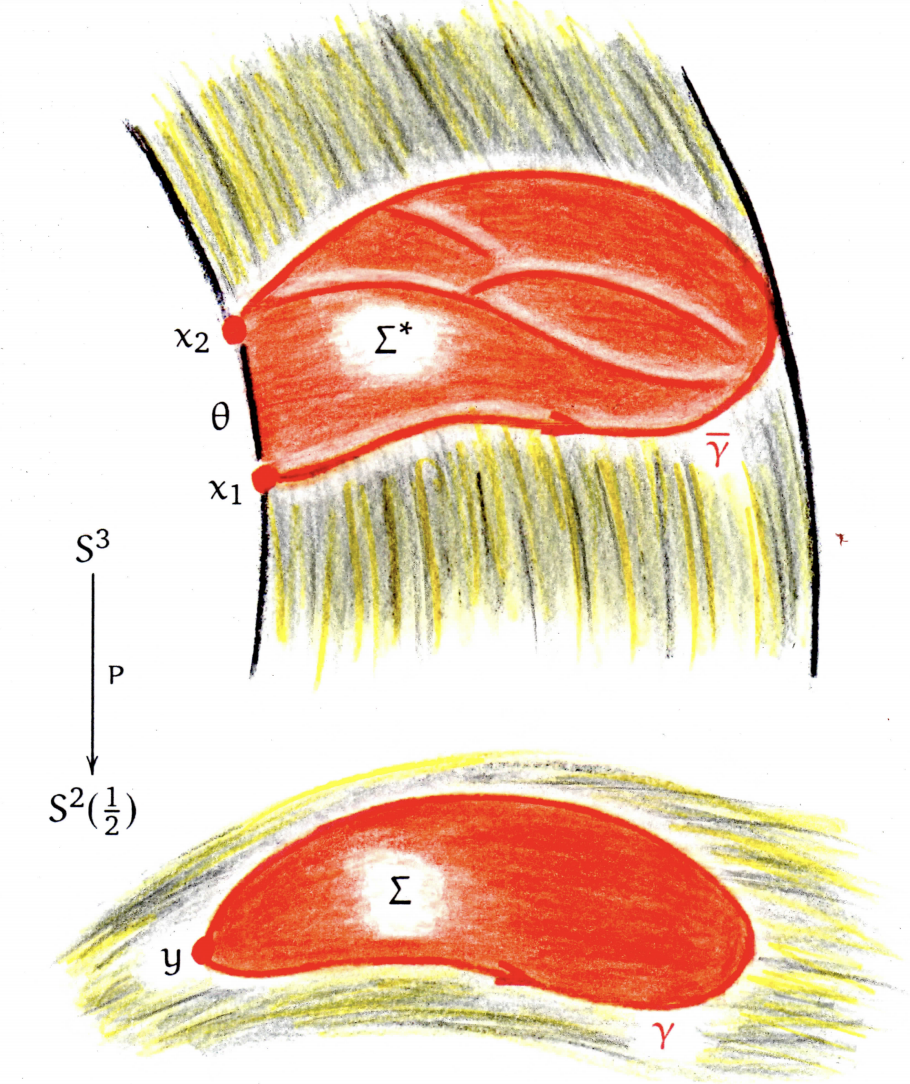}\caption{Holonomy}\label{holonomypic}
\end{center}
\end{figure}

In \autoref{holonomypic} we consider a loop $\gamma$ in $S^2(\frac{1}{2})$ based at the point $y$, and the region $\Sigma$ of $S^2(\frac{1}{2})$ that it bounds. We pick a point $x_1\in p^{-1}(y)$, and consider the horizontal lift $\overline{\gamma}$ of $\gamma$ beginning at $x_1$.\footnote{We warn the reader about the very similar notation for paths, which are denoted by $\gamma$ and the dual form to $C$, which is denoted by $\Upsilon$, since they both appear in this subsection.}

The holonomy here is illustrated by the fact that when the lift $\overline{\gamma}$ returns to the fiber $p^{-1}(y)$, it does so at a point $x_2$ of that fiber, displaced by an angle $\theta$ from the starting point $x_1$. So $\overline{\gamma}$ followed by the arc on $p^{-1}(y)$ from $x_2$ to $x_1$ is a loop in $S^3$. This loop bounds a region $\Sigma^\dast$ in $S^3$, which projects down via $p$ to the region $\Sigma$ on $S^2(\frac{1}{2}).$

We claim that the holonomy angle $\theta$ is given by $$\theta= 2 \text{(area of }\Sigma) \text{ on } S^2({\textstyle \frac{1}{2}}),$$ and confirm this as follows:
$$\text{area of } \Sigma \text{ on } S^2({\textstyle \frac{1}{2}})=\int_\Sigma d(\text{area})=\int_{\Sigma^\dast} p^\dast d(\text{area})=\int_{\Sigma^\dast} \beta \wedge \dualC,$$
using the fact that the Hopf projection $p\colon S^3\to S^2(\frac{1}{2})$ is a Riemannian submersion, and so is area-preserving on the 2-form $\beta\wedge \dualC$, down to the usual area form on $S^2(\frac{1}{2}).$ From \autoref{dualforms} we have that $d\alpha=-2\beta\wedge \dualC$, and hence $d(-\frac{1}{2}\alpha)=\beta \wedge \dualC$. 

Using Stokes' theorem, we get
$$\int_{\Sigma^\dast} \beta\wedge \dualC=\int_{\Sigma^\dast} d(-\frac{1}{2}\alpha)=-\frac{1}{2}\int_{\partial\Sigma^\dast}\alpha.$$

Now $\partial\Sigma^\dast$ consists of two pieces, the arc $\overline{\gamma}$ followed by the arc on $p^{-1}(y)$ from $x_2$ to $x_1$. Since the arc $\overline{\gamma}$ is horizontal, the one-form $\alpha$ is identically zero along it, so we get no contribution to the last integral above. And since the angle along the Hopf great circle $p^{-1}(y)$ measured from $x_1$ to $x_2$ is $\theta$, the integral of $\alpha$ along this arc in the opposite direction is  $-\theta$. 

Putting all this together, we have
$$\text{area of } \Sigma \text{ on } S^2({\textstyle \frac{1}{2}})=  \int_{\Sigma^\dast} \beta \wedge \dualC = -\frac{1}{2}\int_{\partial\Sigma^\dast}\alpha =-\frac{1}{2}(-\theta)=\frac{1}{2}\theta.$$

Hence the holonomy of horizontal transport in $S^3$ induced by the loop $\gamma$ on $S^2(\frac{1}{2})$ is given by the 
$$\text{holonomy angle } \theta = 2\text{ area of } \Sigma \text{ on } S^2({\textstyle \frac{1}{2}}),$$ as claimed above.

\begin{exmp}
The equator $\gamma$ on $S^2(\frac{1}{2})$  bounds a hemisphere $\Sigma$ of area $\frac{\pi}{2}$. The inverse image $p^{-1}(\gamma)$ of $\gamma$ is a Clifford torus in $S^3$, filled with Hopf fibers. The orthogonal trajectories are Hopf fibers of the opposite handedness and are horizontal with respect to the original Hopf fibration. Starting at any location along any original Hopf fiber on this Clifford torus and then following a horizontal circle will bring us back to the antipodal point on the starting fiber. So the holonomy angle in this case is $\theta=\pi$, which is twice the area of $\Sigma$. 
\end{exmp}

\section{Computation of the Lie algebras}\label{appendix1}

In this section, we give an explicit description of the Lie algebras, or equivalently, the tangent spaces at the identity, of the various Fr{\'e}chet Lie groups we consider.

\begin{prop}\label{liealgebras}
The tangent spaces at the identity to our various subgroups of $\Diff(S^3)$ are as follows.
\newline\newline $(a)$ The tangent space $T_{\mathrm{id}}\auth$ consists of vector fields $X = f A + g B + h C$ such that
$$f  = \text{any smooth function on $S^3$},\ \ \ \
\textstyle g = -\frac{1}{2}A h,\ \  \ \ \text{and} \ \ \ \
\textstyle h = \frac{1}{2} A g.$$

\noindent $(b)$ The tangent space $T_{\mathrm{id}}\autxi$ consists of vector fields $X = f A + g B + h C$ such that
$$f  = \text{any smooth function on $S^3$}, \ \ \ \ \
\textstyle g = \frac{1}{2} C f,\ \ \ \ \ \text{and} \ \ \ \ \
\textstyle h = -\frac{1}{2}B f.$$

\noindent $(c)$ The tangent space $T_{\mathrm{id}}\Aut_1(\xi)$ consists of vector fields $X = f A + g B + h C$ such that
$$A f  = 0,\text{ and hence $f$ is constant along fibers of $\HH$},\ \ \ \ \
\textstyle g = \frac{1}{2}C f, \ \ \ \ \ \text{and}\ \ \ \ \
\textstyle h = -\frac{1}{2}B f,$$
and these vector fields are divergence-free.\\

\end{prop}

\vspace{-12pt}

\begin{rem} In view of  \autoref{intersection}, it is natural to ask whether the pair of Lie algebras $T_\id(\auth\cap \Aut(\xi))$ and $T_\id\auth \cap T_\id \Aut(\xi)$ are the same.
The left side is certainly contained in the right side, and we leave it to the reader
to establish the reverse inclusion by manipulating the conditions in parts (a) and (b)
of \autoref{liealgebras}.

\end{rem}

\vspace{-8pt}

We start with an intermediary proposition that gives conditions on the vector fields which are in the tangent spaces of interest. 

\begin{prop}\label{liealgebras0}
The tangent spaces at the identity to our various subgroups of $\Diff(S^3)$ admit the following descriptions.
\begin{enumerate}[label=$(\alph*)$]
\item $T_{\mathrm{id}}\auth = \{X\in VF(S^3) \mid \ L_X A =\lambda A \ \text{ for smooth } \lambda\colon S^3\to \R\}$, \vspace{8pt}
\item $T_{\mathrm{id}}\autxi = \{X\in VF(S^3) \mid \ L_X \alpha=\lambda \alpha \ \text{ for smooth } \lambda\colon S^3\to \R\}$,\vspace{8pt}
\item $T_{\mathrm{id}}\Aut_1(\xi) = \{X\in VF(S^3) \mid \ L_X \alpha=0\}$.
\end{enumerate}
\end{prop}

We will see from the proof of \autoref{liealgebras} that the  functions $\lambda\colon S^3\to \R$ appearing in parts (a) and (b) above have the property that they integrate to zero over each Hopf fiber. Furthermore, for any such function $\lambda$, there exists a vector field $X$ on $S^3$ for which $L_X A=\lambda A$, and similarly there exists a vector field $X$ on $S^3$ for which $L_X \alpha = \lambda \alpha$.

We prove part (a) here. Parts (b) and (c) can be found in \cite[Lemma 1.5.8]{geiges}. Before we delve into the proof, we make some remarks about the definition of Lie derivatives. Let $V$ and $W$ be smooth vector fields on the smooth manifold $M$, let $x\in M$ and let $\{ f_t\}$ be the local one-parameter group generated by $V$, meaning that 
\begin{equation}\label{flow} 
f_0=\id \text{ and for each } x\in M \text{ we have } \left. \frac{df_t(x)}{dt}\right|_{t=0}=V(x).\end{equation}
Then, the \emph{Lie derivative} is traditionally defined as
$$(L_V W)(x)=\lim_{t\to 0}\frac{(f_t^{-1})_\dast W(f_t(x))-W(x)}{t}=
\left. \frac{d}{dt}\right|_{t=0}(f_t^{-1})_\dast W(f_t(x))=\left. \frac{d}{dt}\right|_{t=0}(f_{-t})_\dast W(f_t(x)).$$

In this definition, the one-parameter group $\{f_t\}$ of diffeomorphisms provides the service of pulling the tangent vector $W(f_t(x))$ in the tangent space to $M$ at $f_t(x)$ back to a vector in the tangent space to $M$ at $x$, so that one can subtract from it the tangent vector $W(x)$ living there. But it is easy to check that  any smooth curve $f_t\in \Diff(M)$ satisfying \autoref{flow} can be used to define the Lie derivative $L_V W$ as above, and that requiring $\{f_t\}$ to be a one-parameter subgroup is just a convention, but not essential. Of course, when $\{f_t\}$ is not a one-parameter group, the pullback of $W(f_t(x))$ can only be defined to be $(f_t^{-1})_\dast W(f_t(x))$.

\begin{proof}[Proof of \autoref{liealgebras0}(a)]
Let $X$ be a smooth vector field which lies in $T_\id\auth.$ By definition, this means that there is a smooth curve $f_t$ in $\auth$ with $f_0=\id$ and such that $X(x)=\frac{d}{dt}|_{t=0} f_t(x)$ for all $x\in S^3$. Then, as discussed above, the Lie derivative $L_X A$ is defined as
$$(L_X A)(x)=\lim_{t\to 0}\frac{(f_t^{-1})_\dast A(f_t(x))-A(x)}{t}.$$
But note that here $\{f_t\}$ is a path in $\auth$ with $f_0=\id$, so we can write
$$(f_t)_\dast A(x)=\lambda(x,t) A(f_t(x)),$$ since each $f_t$ takes Hopf fibers to Hopf fibers. 
Therefore in the definition of $L_X A(x)$, for any given $t$, both terms in the numerator are multiples of $A(x)$, so we can factor  $A(x)$ out of the limit, and we get that $L_X A=\lambda A$ for a smooth $\lambda\colon S^3\to \R$.

Conversely, suppose $X$ is a smooth vector field on $S^3$ with $L_X A=\lambda A$ for some smooth function $\lambda\colon S^3\to \R$. Let $\{f_t\}$ be the one-parameter group of diffeomorphisms of $S^3$ generated by the vector field $X$, i.e., $f_0=\id$ and for each $x \in S^3$ we have $\left. \frac{df_t(x)}{dt}\right|_{t=0}=X(x).$ Using the group property of this flow, which says that $f_{s+t}(x)=f_s(f_t(x))$, we compute
$$\left. \frac{d}{dt}\right|_t f_t(x)= \left. \frac{d}{ds}\right|_{s=0} f_{s+t}(x)=\left. \frac{d}{ds}\right|_{s=0} f_s(f_t(x))=X(f_t(x)).$$
Thus $\frac{df_t(x)}{dt}=X(f_t(x))$ holds for all $t$ not just $t=0$.

We need to show that the one-parameter group $\{f_t\}$ lies entirely in $\auth$. Let us use local coordinates $(x,y, \theta)$ in a tubular neighborhood of a Hopf fiber, with $(x,y)\in \R^2$ and $\theta\in S^1$ and with $\frac{\partial}{\partial \theta}$ as the unit vector field along the Hopf fibers.


We write the vector field $X$ in local coordinates as $$X=u(x,y,\theta)\frac{\partial}{\partial x}+v(x,y,\theta)\frac{\partial}{\partial y}+w(x,y,\theta)\frac{\partial}{\partial \theta}.$$ Then we can compute the Lie derivative
\begin{eqnarray*}
L_X A &=& [X,A]=- [A, X]=\textstyle -\left[\displaystyle \frac{\partial}{\partial \theta}\, ,\,  X\right]\\
&=& -\left[\frac{\partial}{\partial \theta}\, , \, u\frac{\partial}{\partial x}+v\frac{\partial}{\partial y}+w\frac{\partial}{\partial \theta}\right]\\
&=&  -\frac{\partial u}{\partial \theta} \frac{\partial}{\partial x} -\frac{\partial v}{\partial \theta} \frac{\partial}{\partial y}-\frac{\partial w}{\partial \theta} \frac{\partial}{\partial \theta}\\
&=& \lambda A= \lambda \frac{\partial}{\partial \theta}.
\end{eqnarray*}
From this we see that $\frac{\partial u}{\partial \theta}=0$ and $ \frac{\partial v}{\partial  \theta}=0$, so the functions $u$ and $v$ only depend on $x$ and $y$ and not on $\theta$. We incorporate this by writing
$$X=u(x,y)\frac{\partial}{\partial x}+v(x,y)\frac{\partial}{\partial y}+w(x,y,z)\frac{\partial}{\partial \theta}.$$
We also note from above that $\frac{\partial w}{\partial \theta}=-\lambda$, which integrates to zero around Hopf circles, and hence $$w(x,y,\theta)=w(x,y, \theta+2\pi).$$
Thus locally the flow $\{ f_t\}$ covers a flow on the $xy$-plane and  takes vertical circles to vertical circles, which tells us that each diffeomorphism $f_t$ takes Hopf circles to Hopf circles, and hence $X\in T_\id\auth$, as desired. 
\end{proof}


Now we turn to the proof of \autoref{liealgebras}, and prove each of its parts separately.

\vspace{-9pt}

\begin{proof}[Proof of \autoref{liealgebras}(a)]
Let $X=fA+gB+hC$ be a smooth vector field on $S^3$, written in terms of the orthonormal basis of left-invariant vector fields $A,B$ and $C$ on $S^3$, following the conventions introduced in \autoref{prelims}.  By \autoref{liealgebras0}(a), $X$ lies in $T_\id \auth$ if and only if $L_X A =\lambda A$ for some smooth real-valued function $\lambda$ on $S^3$. We compute $L_XA$ to see what constraints this conditions imposes on the coefficients $f,g$ and $h$. 

Notationally, we switch from Lie derivatives to Lie brackets and compute
\begin{eqnarray*}
L_X A &=& [X, A]=[fA+gB+hC, A]=[fA,A]+[gB, A]+[hC, A]\\
&=&-[A, fA]-[A, gB]-[A,hC]\\
&=&-(Af)A-f[A,A]-(Ag)B-g[A,B]-(Ah)C-h[A,C]\\
&=&-(Af)A-(Ag)B-g(2C)-(Ah)C-h(-2B)\\
&=&-(Af)A+(2h-Ag)B-(2g+Ah)C,
\end{eqnarray*}
\noindent using the bracket relations from \autoref{brackets}.

Therefore, $X\in T_\id \auth$ if and only if $-Af=\lambda$ for some smooth real-valued function $\lambda$, and $2h-Ag=0$ and $2g+Ah=0$. This completes the proof of \autoref{liealgebras}(a).

\end{proof}

Note in this proof that since $\lambda=-Af$ is the negative of the directional derivative of the coefficient $f$ around a Hopf circle, we see why $\lambda$ must integrate to zero around the Hopf fibers.

\begin{proof}[Proof of \autoref{liealgebras}(b)] Again, let $X=fA+gB+hC$ be a smooth vector field on $S^3$. By \autoref{liealgebras0}(b), $X$ lies in $T_\id\autxi$ if and only if $L_X \alpha=\lambda \alpha$ for some smooth $\lambda$. 

 Suppose $X$  lies in $T_\id \autxi$ so that $L_X \alpha=\lambda \alpha$ for some $\lambda$. Rewrite $\alpha(A)=1$ as $\langle  \alpha, A\rangle=1$ and then differentiate to get 
$$0=L_X\langle \alpha, A\rangle = \langle L_X \alpha, A\rangle +\langle \alpha, L_X A\rangle=\langle \lambda \alpha, A\rangle +\langle \alpha, L_X A\rangle.$$
Thus $$\langle \alpha, L_X A \rangle= -\langle \lambda \alpha, A\rangle =-\lambda\langle \alpha, A\rangle=-\lambda.$$
Using the computation for $L_X A$ from part (a), we get $$\langle \alpha, L_X A\rangle=-Af,$$ thus $Af=\lambda.$ 

Analogously to the computation of $L_X A$ in part (a), we can compute 
\begin{eqnarray*}
L_X B &=& -(Bf+2h)A-(Bg)B+(2f-Bh)C\\
L_X C &=& (-Cf+2g)A-(2f+Cg)B-(Ch)C.
\end{eqnarray*}

Proceeding as before with rewriting the equations $\alpha(B)=0$ and $\alpha(C)=0$  as $\langle \alpha, B\rangle =0$ and $\langle \alpha, C\rangle= 0$, and differentiating, we get
\begin{equation}\label{lxbc}\langle \alpha , L_X B\rangle =0 \ \ \ \text{and}\ \ \ \langle \alpha, L_X C\rangle=0.\end{equation}
Combining with the computations of $L_XB$ and $L_X C$ above, we get that
$$\textstyle h=-\frac{1}{2} Bf\ \ \ \ \text{and}\ \ \ \ g=\frac{1}{2}Cf,$$ as desired. 

Conversely, assuming the coefficients of $X$ satisfy the conditions in \autoref{liealgebras}(b), using the computations of $L_X A$, $L_X B$ and $L_X C$, and working backwards from the computations of  the differentiation of the brackets we get 
$$\langle L_X \alpha, A\rangle=Af, \ \ \ \langle L_X \alpha, B\rangle=0 \ \ \ \text{and} \ \ \ \langle L_X \alpha, C\rangle=0,$$ so $L_X\alpha =(Af) \alpha =\lambda \alpha.$

\end{proof}

\vspace{-15pt}
Note again that $\lambda$ is the directional derivative of the coefficient $f$ around Hopf circles, so we reaffirm the observation made after part (a) that $\lambda$ must integrate to zero around Hopf fibers.


\begin{proof}[Proof of \autoref{liealgebras}(c)] 

Let $X=fA+gB+hC$ be a smooth vector field on $S^3$. By \autoref{liealgebras0}(c), $X$ lies in $T_\id\Aut_1(\xi)$ if and only if $L_X \alpha=0$. 

 Suppose $X$  lies in $T_\id \Aut_1(\xi)$ so that $L_X \alpha=0$. Just as in \autoref{liealgebras}(b), rewriting $\alpha(A)=1$ as $\langle  \alpha, A\rangle=1$ and then differentiating, we get 
$$0=L_X\langle \alpha, A\rangle = \langle L_X \alpha, A\rangle +\langle \alpha, L_X A\rangle=\langle \alpha, L_X A\rangle.$$

But again, by the computation for $L_X A$ from part (a), we have $\langle \alpha, L_X A\rangle=-Af,$ thus $Af=0.$ Just as in part (b), combining the computations for $L_X B$ and $L_X C$ from part (b) with \autoref{lxbc}, we get $$\textstyle h=-\frac{1}{2} Bf\ \ \ \ \text{and}\ \ \ \ g=\frac{1}{2}Cf,$$ as desired.

%
%

Conversely, if we assume that the conditions in \autoref{liealgebras}(c) hold, as we saw in the proof of (b), we get that $L_X\alpha=(Af)\alpha$. Thus if $Af=0$, we immediately get $L_X\alpha=0$, so by \autoref{liealgebras0}(c), $X$ lies in $T_\id \Aut_1(\xi)$.

Lastly, we check that any $X\in T_\id \Aut_1(\xi)$ is divergence free. We have
\begin{eqnarray*}
\mathrm{div} X&=& \textstyle Af+Bg+Ch=0+B(\frac{1}{2}Cf)+C(-\frac{1}{2}Bf)\\
&=&  \textstyle \frac{1}{2}(BC-CB)f=\frac{1}{2}[B,C]f=\frac{1}{2} (2A)f=Af=0.
\end{eqnarray*}
\end{proof}

\begin{rem}
The conditions on the coefficients of $X=fA+gB+hC$ in \autoref{liealgebras} may seem mysterious at first glance, and it is a rewarding exercise to try to decode their geometric meaning. We give some hints. In part (a),  you can take the
conditions on the coefficients $g$ and $h$ and differentiate again in the $A$-direction
to show that as the flow of $A$ moves a Hopf fiber off itself, it assumes a coiling shape so as to approximate a nearby Hopf fiber. In part (b), another approach to describing $T_\id \Aut(\xi)$ is to observe that a vector field $X$ is in this space if and only if $L_X B$ and $L_X C$ both lie in the 2-plane spanned by $B$ and $C$, and then compute with Lie brackets.
\end{rem}


Having given in \autoref{liealgebras} a description of the tangent space at the identity to our various subgroups of $\Diff(S^3)$, we end \autoref{appendix1} now with a similar description of the tangent spaces $T_\id\Diff(S^2)$ and  $T_\id \SDiff(S^2)$. 

We can  doubly appreciate our ability to write vector fields on $S^3$ in terms of left-invariant vector fields $A, B, $ and $C$ when we turn to $S^2$ and seek a similar description there. But we can use the Hopf projection $p\colon S^3\to S^2$ to uniquely lift smooth vector fields on $S^2$ to smooth horizontal fields on $S^3$, that is, vector fields which are orthogonal to the Hopf fibers and, by virtue of lifting from $S^2$, twist around each Hopf fiber so they lie in $T_\id \auth$. This allows us to think of $T_\id \Diff(S^2)$ and $T_\id \SDiff(S^2)$ as subspaces of $T_\id \auth$, and therefore rely on expressions in terms of $A,B$ and $C$ to describe the vector fields therein. With this identification in mind, we prove the following proposition.

\begin{prop}\label{TSDiff} The tangent spaces  $T_{\mathrm{id}}\Diff(S^2)$ and  $T_{\mathrm{id}}\SDiff(S^2)$ have the following descriptions.
\newline\newline $(a)$ The tangent space $T_{\mathrm{id}}\Diff(S^2)$ consists of vector fields $X = f A + g B + h C$ such that
$$ f = 0, \ \ \ \ \
\textstyle g = -\frac{1}{2}A h, \ \ \ \ \ \text{and}\ \ \ \ \
\textstyle h = \frac{1}{2}A g. $$

\noindent $(b)$ The tangent space $T_{\mathrm{id}}\SDiff(S^2)$ consists of vector fields $X = f A + g B + h C$ such that
$$ f = 0, \ \ \ \ \
\textstyle g = -\frac{1}{2}A h, \ \ \ \ \
\textstyle h = \frac{1}{2}A g, \ \ \ \ \text{and}\ \ \ \ \
Bg+Ch =0.$$
\end{prop}

\begin{proof}
For part (a), note that we know from \autoref{liealgebras}(a) that the tangent space $T_\id \auth$ consists of  vector fields $X=fA+gB+hC$   such that $f$ is any smooth function on $S^3$, $g = -\frac{1}{2}A h$, and $h = \frac{1}{2} A g.$ If $X$ is horizontal, then $f=0$. Thus the conditions in part (a) are certainly necessary for $X$ to be the horizontal lift of a vector field in $T_\id \Diff(S^2)$.

Conversely, suppose that a vector field $X$ on $S^3$ satisfies the conditions in part (a), and since $f=0$, write $X=gB+hC$. 

We claim that the horizontal vector field $X=gB+hC$ is the lift of a vector field $X'$ on $S^2$ if and only if $L_A X=0$. To see this, note that the left-invariant vector field $A$ on $S^3$ is the infinitesimal generator of the one parameter subgroup of $\auth$ consisting of diffeomorphisms  of $S^3$, $\mathrm{rot}_\theta\colon x\mapsto xe^{i\theta}$ for $0\leq \theta \leq 2\pi$, which uniformly rotate all Hopf fibers by the same amount. Then the horizontal vector field $X$ is the lift of a vector field on $S^2$ if and only if $(\mathrm{rot}_\theta)_\dast  X(x)=X(xe^{i\theta})$, which is equivalent to $L_A X=0$.

From our computation of $L_A X$ in the proof of \autoref{liealgebras}(a), and setting $f=0$, we have
$$L_A X= (Ag-2h)B+(Ah+2g)C,$$ which is equal to $0$ by our conditions in part (a). Thus $X$ is the lift of a vector field on $S^2$. So the stated conditions are both necessary and  sufficient for $X$ to lie in $T_\id \Diff(S^2)$.

For part (b) of our current proposition, it is easy to check that $T_\id \SDiff(S^2)$ consists of all divergence-free vector fields on $S^2$. We claim that  a vector field $X'\in T_\id \Diff(S^2)$ is divergence-free if and only if its horizontal lift to $X=gB+hC\in T_\id \auth$ is divergence-free, which in turn is equivalent to the condition that $Bg+Ch$=0. This will show that the extra condition in part (b) is both necessary and sufficient for a vector field $X$ from part (a) to actually lie in the subspace $T_\id \SDiff(S^2)$.

To prove the claim, let $\{\varphi'_t\}$ be the one-parameter group of diffeomorphisms of $S^2$ generated by the vector field $X'$, and let $\{\varphi_t\}$ be their lifts to a one-parameter group of diffeomorphisms of $S^3$ generated by the lifted vector field $X$.  

If we assume that the lifted field $X$ is divergence-free, then the diffeomorphisms $\varphi_t$ are volume-preserving on $S^3$. Moreover, since $X$ is orthogonal to the Hopf fibers, the diffeomorphisms $\varphi_t$ take Hopf fibers rigidly to Hopf fibers. It then follows that the diffeomorphisms $\varphi'_t$ must be area-preserving on $S^2$ and their generating vector field $X'$ must be divergence-free on $S^2$. 

Conversely, if we assume that the vector field $X'$ on $S^2$ is divergence-free, it follows that the diffeomorphisms $\varphi'_t$ are area-preserving there. Then, since the horizontally lifted vector field $X$ on $S^3$ is the infinitesimal generator of a one-parameter subgroup of diffeomorphisms   $\varphi_t$ of $S^3$ which take Hopf fibers rigidly to one another, and which cover the area-preserving diffeomorphisms  $\varphi'_t$ of $S^2$, the diffeomorphisms  $\varphi_t$ must be volume-preserving on $S^3$ and hence the vector field $X$ must be divergence-free there. 

This proves the claim, and completes the proof of \autoref{TSDiff}.

\end{proof}

\section{The exact sequence of Lie algebras}\label{lieexactsec}

In this section, we establish the exactness of the following sequence on the level of Lie algebras. We note that this result also appears in \cite{ratiu_schmid}, where Ratiu and Schmid attribute it to \cite{konstant}, but give their own proof. We give our own version of a proof here, building on our explicit computation from the previous section.

\begin{prop}\label{Lieseq}
The sequence of tangent spaces
$$0\to  T_\id S^1 \xrightarrow{\TJ} T_\id \Aut_1(\xi) \xrightarrow{\TP} T_\id \SDiff(S^2)\to 0$$
is an exact sequence of Lie algebras.
\end{prop} 


Before turning to the proof, we give explicit descriptions of the tangent spaces in the sequence, which are computed in detail in \autoref{appendix1}. Writing a smooth vector field on $S^3$ as $X=fA+gB+hC$ as in \autoref{prelims}, the conditions on the coefficients $f,g$ and $h$, which describe membership in the tangent spaces in question are as follows:
\begin{enumerate}
\item $X\in T_\id S^1$ if and only if
$$ f  = \text{constant},  \ \ \ g = 0,\ \ \ h = 0,$$
\item $X\in T_\id \Aut_1(\xi)$ if and only if
$$ \textstyle A f  = 0,\ \ \ g = \frac{1}{2} C f ,\ \ \ h = -\frac{1}{2}B f.$$
\end{enumerate}

We view $T_\id \SDiff(S^2)$ as horizontal vector fields on $S^3$, which push forward consistently along Hopf fibers to divergence-free vector fields on $S^2$, where by ``consistently'' we mean that $p_\dast(X)|_x=p_\dast(X)|_y$ for all $x,y$ in the same Hopf fiber. With this interpretation, we get the following description.

\begin{enumerate}
\setcounter{enumi}{2}
\item $X\in T_\id \SDiff(S^2)$ if and only if
$$\textstyle  f  = 0,\ \ \ g = -\frac{1}{2} Ah, \ \ \  h = \frac{1}{2} Ag,\ \ \ Bg+Ch = 0.$$
\end{enumerate}


It is easy to see (1), whereas (2) is proved as part (c) of \autoref{liealgebras} and (3) is \autoref{TSDiff}.

\begin{proof}[Proof of \autoref{Lieseq}]
We start by showing that the maps $\TJ$ and $\TP$ do restrict to maps between tangent spaces. First, in order for $fA\in T_\id S^1$, $f$ must be constant, so we have $\TJ(fA)=fA\in T_\id \Aut_1(\xi)$. 

For $fA+gB+hC\in T_\id\Aut_1(\xi)$, we have $\TP(fA+gB+hC)=0A+gB+hC$. To show that this lives in $T_\id \SDiff(S^2)$, we need to verify that if $f,g,$ and $h$ satisfy the conditions in (2), then $g,h$ satisfy the conditions in (3). Using the description of $g$ and $h$ from (2), note that the condition $g=-\frac{1}{2}Ah$ is equivalent to $2Cf=ABf.$ This equality can be seen to be true using the bracket formula $2C=AB-BA$ and the fact that $Af$ is also assumed to be 0. In a similar fashion, we can show that $h=\frac{1}{2} Ag$. Lastly, again using the description of $g$ and $h$ from (2), we get that $Bg+Ch=\frac{1}{2}(BC-CB)f=Af=0$.

Now we turn to  exactness of the sequence. The map $\TJ$ is injective, so we have exactness at $T_\id S^1$. To see exactness at $T_\id \Aut_1(\xi)$, first note that by definition it follows immediately that $\mathrm{im}(\TJ)\subseteq \mathrm{ker}(\TP)$. To see the reverse inclusion, suppose $X=fA+gB+hC$ and suppose $\TP(X)=gB+hC=0$. Then $g=\frac{1}{2} Cf=0$ and $h=-\frac{1}{2} Bf=0$. But then $Af=\frac{1}{2}(BC-CB)f=0.$ Thus $f$ is constant on $S^3$, and $X=fA\in \mathrm{im}(\TJ)$.

Lastly, to verify exactness at $T_\id \SDiff(S^2)$ we need to check that $\TP$ is surjective. Suppose that $Y=gB+hC\in T_\id \SDiff(S^2)$, so the coefficients satisfy the conditions in (3). We need to find a smooth function $f\colon S^3\to \R$ such that the vector field $X=fA+gB+hC$ lies in $T_\id \Aut_1(\xi)$, i.e., so that $f,g$ and $h$ satisfy the equations in (2). Combining the conditions on $f,g$ and $h$ from (2) and (3), we have $Af=0$, $Bf=-Ag$ and $Cf=-Ah$.

Plugging this into the gradient formula from \autoref{grad}, we are seeking $f$ so that
$$\mathrm{grad}(f) =(Af)A+(Bf)B+(Cf)C=-(Ag)B-(Ah)C.$$

On $S^3$ we can solve for $f$ if and only if $\mathrm{curl} \big((-Ag)B+(-Ah)C\big)=0$. From \autoref{curl}, after simplifying, we get
$$\mathrm{curl}  \big((-Ag)B+(-Ah)C\big) = (CAg-BAh)A+(2Ag+A^2h)B+(2Ah-A^2g)C.$$
Differentiating the equations $g=-\frac{1}{2}Ah$ and $h=\frac{1}{2}Ag$ with respect to $A$, we get $A^2h=-2Ag$ and $A^2g=2Ah$, thus our equation reduces to
$$\mathrm{curl}  \big((-Ag)B+(-Ah)C\big)= (CAg-BAh)A.$$
Furthermore, using the equations $Ag=2h$, $Ah=-2g$ and $Ch+Bg=0$, we conclude that $\mathrm{curl}  \big((-Ag)B+(-Ah)C\big)=0$, and thus $(-Ag)B+(-Ah)C=\mathrm{grad}(f)$ for some smooth function $f\colon S^3\to \R$, as desired. This completes the proof of the proposition, namely that our sequence of tangent spaces is exact.
\end{proof}

\vspace{20pt}
\section{The exact sequence of Fr{\'e}chet Lie groups}\label{groupexactsec}
%

\noindent In this section we establish the exactness of the sequence on the level of Lie groups. More precisely, we give an independent proof of the following theorem, originally due to Banyaga \cite{banyaga1, banyaga2}, Souriau \cite{souriau} and Kostant \cite{konstant}.

\vspace{10pt}
\begin{thm}\label{thm_exactness_groups}
	The sequence of Fr\'echet Lie groups 
	\begin{equation}\label{eq_exactness_groups}
	\{1\} \rightarrow S^1 \overset{\TJ}\longrightarrow \mathrm{Aut}_1(\xi) \overset{P}{\longrightarrow} \SDiff(S^2) \rightarrow \{1\} 
	\end{equation}
	is exact.
\end{thm}
The $S^1$ subgroup in the above exact sequence is the set of diffeomorphisms which rotate the Hopf fibers within themselves by the same angle. The projection  $P \colon \mathrm{Aut}_1(\xi) \rightarrow \SDiff(S^2)$ starts with a diffeomorphism $F$ in $\mathrm{Aut}_1(\xi)$ and then records the resulting permutation of the Hopf fibers. We can write
\begin{equation}
P(F)(y) = p \circ F \circ p^{-1}(y)
\end{equation}
where $p\colon S^3 \rightarrow S^2$ is the Hopf map.

The proof of  \autoref{thm_exactness_groups} is broken down into two lemmas, corresponding to the  two main challenges: proving that the kernel of $P$ is no larger than the subgroup $S^1$, and proving that the map $P$ is onto $\SDiff(S^2)$. The map from $S^1$ into $\mathrm{Aut}_1(\xi)$ is just the inclusion, so exactness there is automatic.

\begin{lem}
	The sequence from \autoref{eq_exactness_groups} is exact at $\mathrm{Aut}_1(\xi)$.
\end{lem}

\begin{proof}
	The map $P$ takes the subgroup $S^1$ of $\mathrm{Aut}_1(\xi)$ to the identity of $\SDiff(S^2)$, because the elements of this subgroup just rotate the fibers within themselves, and so induce the identity map of $S^2$ to itself. Thus, to confirm exactness at $\mathrm{Aut}_1(\xi)$, the challenge is to show that the kernel of $P$ is no larger than this subgroup.
	
	\begin{figure}[h!]
\begin{center}
	\includegraphics[scale=0.45]{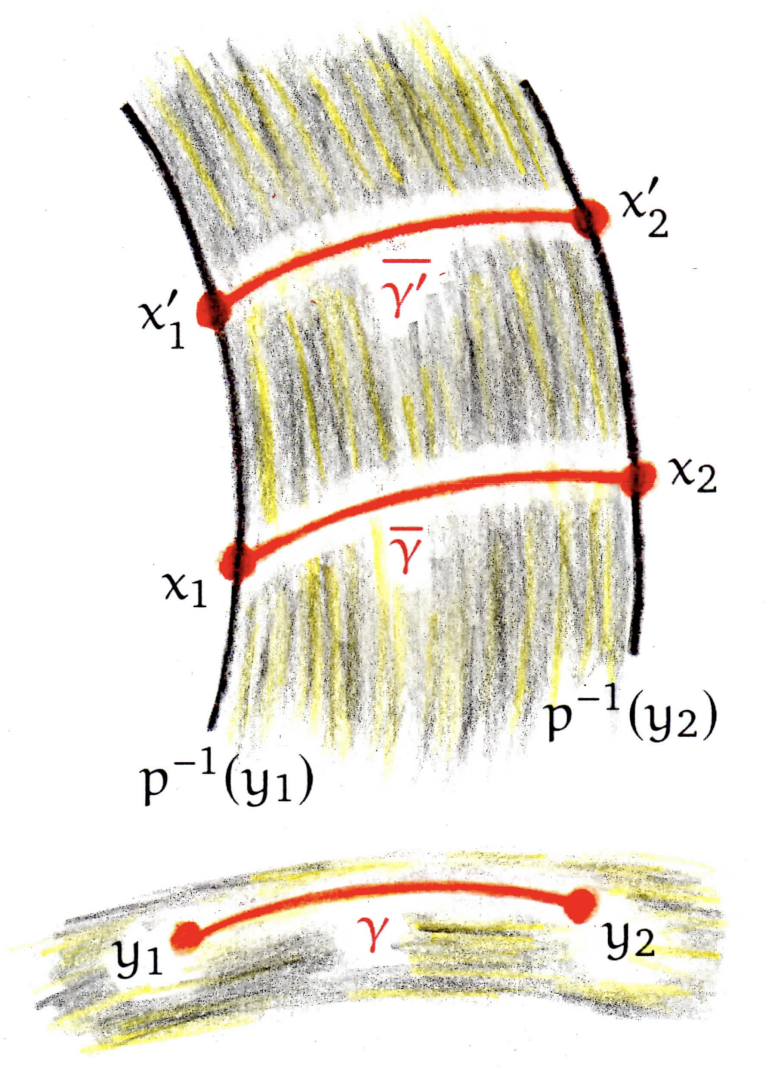}\caption{Exactness at $\Aut_1(\xi)$}\label{exactaut1pic}
\end{center}
\end{figure}
	
	We start with an element $F \in \mathrm{Aut}_1(\xi)$ which takes each Hopf fiber rigidly to itself, and show that it rotates each fiber within itself by the same amount.
	
	We consider two Hopf fibers $p^{-1}(y_1)$ and $p^{-1}(y_2)$, and connect the points $y_1$ and $y_2$ of $S^2$ by a geodesic arc $\gamma$ there. We can assume these points $y_1$ and $y_2$ are not antipodal, since we only need to show that the
amount each Hopf fiber is rotated by $F$ is locally constant. With this choice, the geodesic
arc $\gamma$ connecting $y_1$ and $y_2$ is unique, and we have a well-defined nearest neighbor map
between $p^{-1}(y_1)$ and $p^{-1}(y_2)$.
	
	Then we choose two points $x_1$ and $x_1'$ on the fiber $p^{-1}(y_1)$, and consider the two horizontal lifts $\overline{\gamma}$ and $\overline{\gamma}'$  of $\gamma$ which begin at $x_1$ and $x_1'$. These horizontal lifts are geodesics in $S^3$, and they end on the fiber $p^{-1}(y_2)$ at the points $x_2$ and $x_2'$ which are the nearest neighbors there to the points $x_1$ and $x_1'$, respectively, on $p^{-1}(y_1)$.
		
	Since the nearest neighbor map from $p^{-1}(y_1)$ to $p^{-1}(y_2)$ is an isometry between  Hopf fibers, the angle $\theta$ between $x_1$ and $x_1'$ on the first fiber is the same as the angle $\theta$ between $x_2$ and $x_2'$ on the second fiber.
	
	Now given $x_1 \in p^{-1}(y_1)$, we choose $x_1'$ to be $F(x_1)$. Since $F$ is a contactomorphism, it permutes the contact tangent $2$-planes $\xi = \ker(\alpha)$ among themselves, and so in particular takes horizontal curves to horizontal curves in $S^3$.
	
	It follows that $F(\overline{\gamma}) = \overline{\gamma}'$, and in particular $F(x_2) = x_2'$. This means that the angle $\theta$ between the points $x_1$ and $x_1' = F(x_1)$ on the Hopf fiber $p^{-1}(y_1)$ is the same as the angle $\theta$ between the points $x_2$ and $x_2' = F(x_2)$ on the Hopf fiber $p^{-1}(y_2)$. Thus, $F$ rotates all fibers by the same amount, which means that $F \in S^1$, which is what we wanted to prove. This confirms exactness of our sequence of Fr{\'e}chet Lie groups at $\Aut_1(\xi)$.
\end{proof}

We  turn now to exactness at $\SDiff(S^2)$, following the approach introduced by Ratiu and Schmid in \cite{ratiu_schmid}. Given the Hopf projection $p \colon S^3 \rightarrow S^2$ and a path $\gamma$ in $S^2$, we denote by
\begin{equation}\label{eq_horizontal_transport}
H_{\gamma} \colon p^{-1}(\gamma(0)) \rightarrow p^{-1}(\gamma(1))
\end{equation}
the \emph{horizontal transport} along $\gamma$, in which each point of the first fiber moves along the horizontal lift of $\gamma$ to a point on the second fiber, as introduced in \autoref{horizontal}. This rigid motion between great circle fibers is the continuous analog of our nearest neighbor maps. Recall from \autoref{horizontal} that if the path $\gamma$ in $S^2$ is a geodesic arc, then the map in \autoref{eq_horizontal_transport} is precisely the nearest neighbor map between these two Hopf fibers.

Recall that the subgroup $\Aut_1(\HH)$ of \emph{strict automorphisms of $\HH$} is the subgroup of $\auth$ permuting Hopf fibers rigidly,
$$\Aut_1(\HH)=\{F\in\Diff(S^3)\ \mid \  F_\dast A = A \}.$$
The following lemma characterizes the strict automorphisms of the Hopf fibration which commute with horizontal transport.

\begin{lem}\label{lemma_aut_1}
	Let $F \in \mathrm{Aut}_1(\HH)$ induce $f \in \SDiff(S^2)$ through $f(y) = p \circ F \circ p^{-1}(y)$. Then $F \in \mathrm{Aut}_1(\xi)$ if and only if
	\begin{equation}
	F \circ H_{\gamma} = H_{f  \gamma} \circ F
	\end{equation}
	for all smooth curves $\gamma$ in $S^2$.
\end{lem}

\begin{figure}[h!]
\begin{center}
	\includegraphics[scale=0.45]{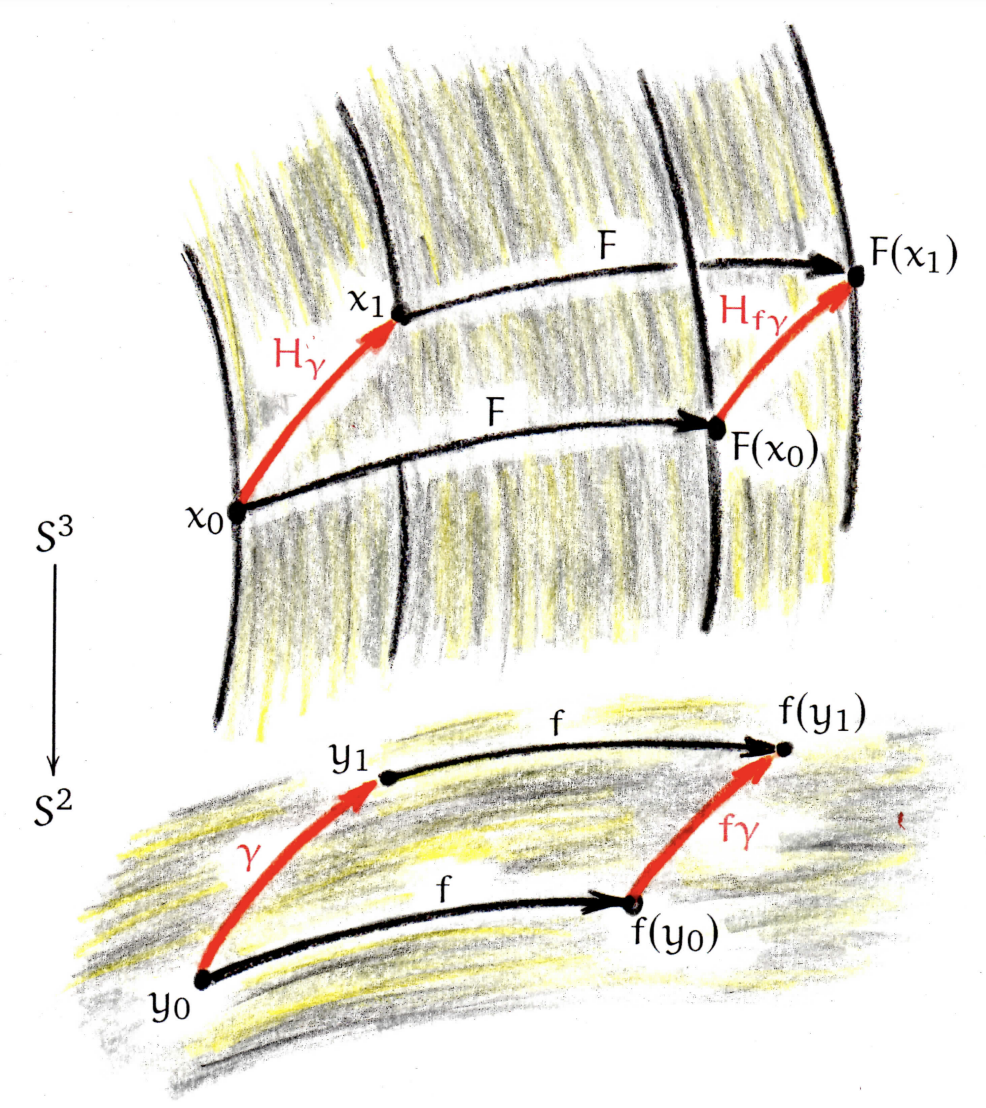}\caption{Horizontal transport}\label{transportpic}
\end{center}
\end{figure}

\vspace{10pt}

\begin{proof}
	If $F \in \mathrm{Aut}_1(\xi)$, then $F$ takes horizontal curves in $S^3$ to horizontal curves. In particular, in \autoref{transportpic}, $F$ takes the horizontal curve labeled $H_{\gamma}$, which runs from $x_0$ to $x_1$, to the horizontal curve labeled $H_{f\gamma}$, which runs from $F(x_0)$ to $F(x_1)$. Thus, $F \circ H_{\gamma} = H_{f\gamma} \circ F$.
	
	Conversely, suppose that $F \circ H_{\gamma} = H_{f\gamma} \circ F$ for all smooth curves $\gamma$ in $S^2$. Then given any point $x \in S^3$, choose two horizontal curves through $x$ whose tangent vectors at $x$ span the tangent $2$-plane $\xi_x$. Since $F$ takes horizontal curves in $S^3$ to horizontal curves, its differential $dF(x)$ must take $\xi_x$ to $\xi_{F(x)}$, which means $F \in \mathrm{Aut}(\xi)$. Since we started out with $F \in \mathrm{Aut}_1(\HH)$, we have $F \in \mathrm{Aut}_1(\HH) \cap \mathrm{Aut}(\xi) = \mathrm{Aut}_1(\xi)$. 
\end{proof}

\vspace{5pt}

\begin{lem}\label{lemma_5_8}
	The sequence of Fr{\'e}chet Lie groups 
	\begin{equation*}
	\{1\} \rightarrow S^1 \rightarrow \mathrm{Aut}_1(\xi) \overset{P}{\longrightarrow} \SDiff(S^2) \rightarrow \{1\} 
	\end{equation*}
	from \autoref{eq_exactness_groups} is exact at $\SDiff(S^2)$. That is, the map $P\colon \Aut_1(\xi)\to \SDiff(S^2)$ is onto.
\end{lem}

\begin{proof}
	We start out with a diffeomorphism $f \in \SDiff(S^2)$, which we want to lift to an automorphism $F \in \mathrm{Aut}_1(S^3)$.
	
		\begin{figure}[h!]
	\begin{center}
		\includegraphics[scale=0.41]{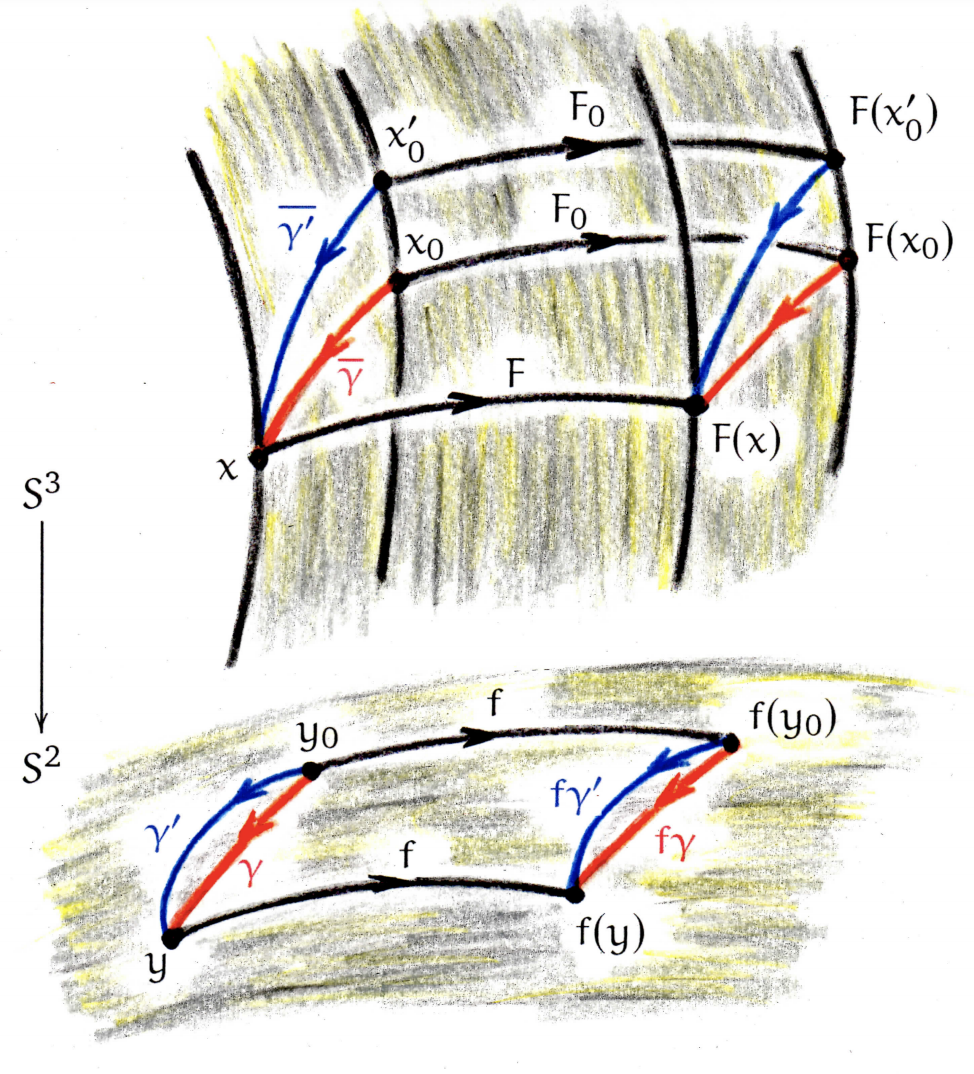}\caption{Path lifting}\label{pathliftpic}
	\end{center}
	\end{figure}
	
	We fix a point $y_0 \in S^2$ to serve as our base point throughout the proof and then begin the definition of the diffeomorphism $F$ of $S^3$ by requiring that it  take the Hopf fiber $p^{-1}(y_0)$ rigidly to the Hopf fiber $p^{-1}(f(y_0))$ in an orientation-preserving but otherwise arbitrary way. We let
	\begin{equation}
	F_0 \colon p^{-1}(y_0) \rightarrow p^{-1}(f(y_0))
	\end{equation}
	be this map, which is determined up to a rigid rotation.

	Next, consider an arbitrary point $x \in S^3$ and its projection $y = p(x)$ in $S^2$. We connect $y_0$ and $y$ with an arbitrary smooth path $\gamma$ in $S^2$, so that $\gamma(0) = y_0$ and $\gamma(1) = y$, and let $\overline{\gamma}$ denote its unique horizontal lift to a path in $S^3$ which \emph{ends} at $x$, meaning $\overline{\gamma}(1) = x$, as in \autoref{pathliftpic}.
	
Let $x_0 = \overline{\gamma}(0)$ be the beginning point of this lifted path, so that $x_0$ lies somewhere on the Hopf fiber $p^{-1}(y_0)$. In the notation of horizontal transport, we can write $x_0 = H_{\gamma}^{-1}(x)$. The diffeomorphism $F_0$ has already been defined on this ``base'' Hopf fiber, so we know the point $F_0(x_0)$.

Now consider the smooth path $f(\gamma)$ in $S^2$, which runs from $f(y_0)$ to $f(y)$. The unique horizontal lift of this path which begins at $F_0(x_0)$ is shown in the figure. Horizontal transport in $S^3$ along this horizontal lift takes the point $F_0(x_0)$ to the point that we will define to be $F(x)$, that is,
\begin{equation}\label{eq_definition_F}
F(x) = H_{f\gamma} \circ F_0 \circ H_{\gamma}^{-1}(x)
\end{equation}
We will show that the definition of $F$ does not depend on the choice of the path $\gamma$ from $y_0$ to $y$ in $S^2$, and this will follow from the fact that the diffeomorphism $f$ of $S^2$ is area-preserving. To that end, let $\gamma'$ be another smooth path in $S^2$ from $y_0$ to $y$, shown in \autoref{pathliftpic}. 


We must show that
\begin{equation}\label{eq_F_welldefined}
H_{f\gamma'} \circ F_0 \circ H_{\gamma'}^{-1} = H_{f\gamma} \circ F_0 \circ H_{\gamma}^{-1}
\end{equation}
Consider the loop $\sigma = \gamma (\gamma')^{-1}$ in $S^2$ based at $y_0$ that runs through $\gamma$ and then $\gamma'$ backwards. The image under $f$ of this loop is the loop $f\sigma = (f\gamma)(f\gamma')^{-1}$ based at $f(y_0)$. Then a little transposing of terms in \autoref{eq_F_welldefined} gives us
\begin{equation}\label{eq_h_sigma}
 F_0\circ H_{\sigma}   =  H_{f\sigma}\circ F_0
\end{equation}
Since $f$ is area-preserving, the areas enclosed by the loops $\sigma$ and $f\sigma$ are the same. Hence, by the results of \autoref{quantitativeholonomy}, the holonomy experienced by the horizontal lifts of these loops are equal, and preserved by the rigid motion $F_0$ between the fibers. This confirms  \autoref{eq_h_sigma}, and hence that $F$ does not depend on the choice of the path $\gamma$ in $S^2$ running from $y_0$ to $y$. A different choice of basepoint $y_0^\dast$ in $S^2$ in this construction would result in a new map $F^\dast$ which differs from $F$ by a uniform rotation on all Hopf fibers.

We note that by construction $F$ covers $f$, i.e., $P \circ F = f \circ P$. \autoref{eq_definition_F}, which defines $F$, together with  \autoref{lemma_aut_1} show that $F$ is in $\mathrm{Aut}_1(\xi)$. Since $F$ takes Hopf fibers rigidly to Hopf fibers and covers the diffeomorphism $f$, its differential $dF(x)$ at each point $x \in S^3$  cannot have a nontrivial kernel. Hence $F$ is a submersion from $S^3$ to itself, thus a covering map, and since $S^3$ is simply connected, $F$ is a diffeomorphism. We leave the proof of smoothness of $F$ for \autoref{appendix_frechet}.
\end{proof}

This concludes the proof of exactness of the sequence of Fr\'echet Lie groups stated in  \autoref{thm_exactness_groups}.

\section{The fiber bundle structure}\label{bundlesec}


\noindent The goal of this section is to give an independent proof of  the following theorem, originally due to Vizman \cite{vizman}.

\begin{thm}\label{fiberbundle}
The sequence
\begin{equation}\label{eq_fiber_bundle}
S^1 \hookrightarrow \mathrm{Aut}_1(\xi) \overset{P}{\longrightarrow} \SDiff(S^2)
\end{equation}
is a fiber bundle in the Fr\'echet category. 
\end{thm} 

\begin{proof} This amounts to constructing slices over small open sets in $\SDiff(S^2)$, and then using the action of the subgroup $S^1$ to promote these slices to the product neighborhood needed to confirm the bundle structure.

First, we note that $F$, which was defined by the formula
\begin{equation}
F(x) = H_{f\gamma} \circ F_0 \circ H_{\gamma}^{-1}(x)
\end{equation}
depends smoothly on $f \in \SDiff(S^2)$. This follows from the fact that the composition map
\begin{gather}
\begin{split}
\circ \colon \SDiff(S^2) \times \mathrm{Path}(S^2) &\rightarrow \mathrm{Path}(S^2) \\
(f, \gamma) &\mapsto f \circ \gamma 
\end{split}
\end{gather}
is smooth in the Fr\'echet category, together with the fact that $F$ is smooth as a function of $x\in S^3$, $\gamma\in \mathrm{Path}(S^2)$ and $f\in \SDiff(S^2)$ (see \autoref{prop_path_sn} and \autoref{prop_smoothness_F}). 

Second, we restrict attention to a small neighborhood of the identity $\mathrm{id} \in \SDiff(S^2)$, for example the set
\begin{equation}
U = \big\{ f \in \SDiff(S^2) \, : \, d(y, f(y)) < \pi/4,~~\forall y \in S^2 \big\},
\end{equation}
where we regard $S^2$ as the sphere of radius $\frac{1}{2}$ so that the Hopf projection $p \colon S^3 \rightarrow S^2$ is a Riemannian submersion. Restricting $f$ to this open set $U$ will let us uniquely define the nearest neighbor map from $p^{-1}(y_0)$ to $p^{-1}(f(y_0))$ to serve as the map $F_0$.

To construct our slice, define $\varphi \colon U \rightarrow \mathrm{Aut}_1(\xi)$ by
\begin{equation}
\varphi(f) = F,~\text{where $F$ is the map}~ F(x) = H_{f\gamma} \circ F_0 \circ H_{\gamma}^{-1}.
\end{equation}
Note that the nearest neighbor map $F_0 \colon p^{-1}(y_0) \rightarrow p^{-1}(f(y_0))$ between Hopf fibers depends smoothly on $f$ \cite{Eells}, and $\gamma$ is chosen as the (unique) shortest geodesic connecting $y_0$ and $f(y_0)$, which is possible since $f \in U$.

Hence $\varphi \colon U \rightarrow \mathrm{Aut}_1(\xi)$ is a smooth map of Fr\'echet manifolds, with
\begin{equation}
P \circ \varphi = \mathrm{id_U} \colon U \rightarrow U.
\end{equation}
This is the slice over $U$ for the proposed bundle \eqref{eq_fiber_bundle}. We now promote this slice to a product neighborhood in $\mathrm{Aut}_1(\xi)$ over $U$ by using the action of the circle group $S^1$ as follows. Let
\begin{gather}\label{eq_map_slice}
\begin{split}
\Phi \colon S^1 \times U &\rightarrow \mathrm{Aut}_1(\xi) \\
(\theta, f) &\mapsto e^{i\theta}\varphi(f) = e^{i\theta}F
\end{split}
\end{gather}
where the right hand side takes the element $\varphi(f)$ of $\mathrm{Aut}_1(\xi)$ and either follows or precedes it (same result) by uniformly rotating all Hopf fibers through the angle $\theta$. Since multiplication in the Fr\'echet Lie group $\mathrm{Aut}_1(\xi)$ is smooth, it follows that \eqref{eq_map_slice} is a smooth map of Fr\'echet manifolds. To check that it gives the local product structure required to confirm that \eqref{eq_fiber_bundle} is a Fr\'echet fiber bundle, we write down its inverse $\Phi^{-1}$ explicitly and check that it is also smooth.

\vspace{-20pt}
	\begin{figure}[h!]
	\begin{center}
		\includegraphics[scale=0.5]{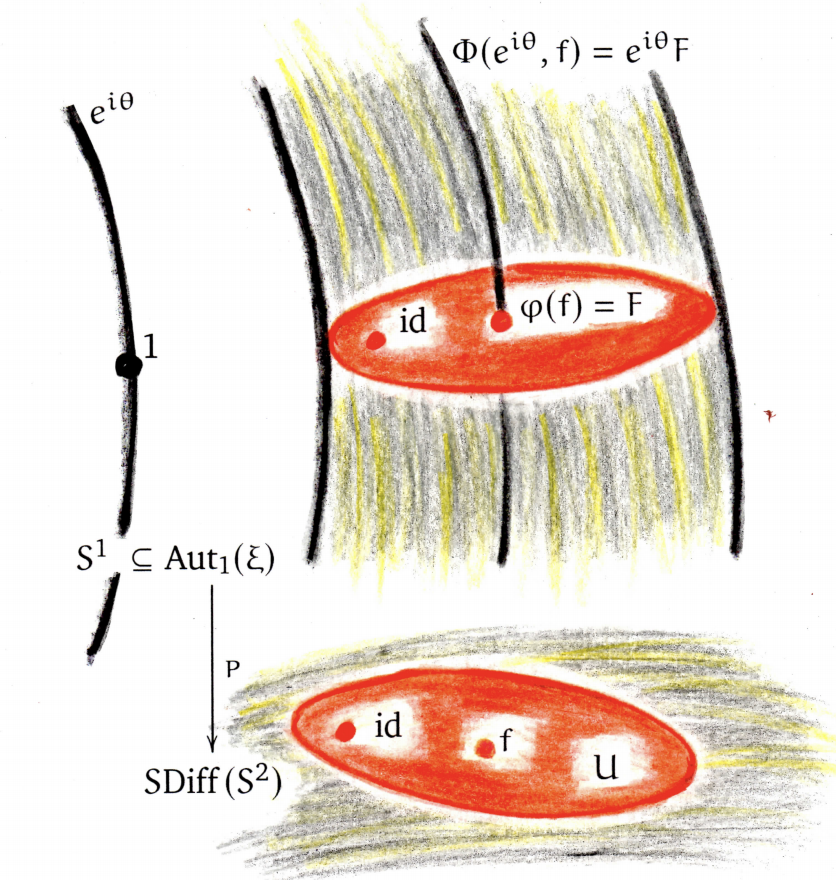}\caption{Local product structure}\label{localprodpic}
	\end{center}
	\end{figure}
	
\vspace{-10pt}

To define $\Phi^{-1} \colon P^{-1}(U) \rightarrow S^1 \times U$, let $G$ be any diffeomorphism of $S^3$ lying in the tube $P^{-1}(U) \subseteq \mathrm{Aut}_1(\xi)$ and let $f = P(G) \in U$. Then define $F = \varphi(f)$, and since $P(F) = f$, the diffeomorphisms $G$ and $F$ lie in the same circular fiber $P^{-1}(f)$, separated by some angle $\theta$. We identify this angle by $\theta = GF^{-1}$. Define
\begin{gather}\label{eq_map_slice2}
\begin{split}
\Phi^{-1} \colon P^{-1}(U) &\rightarrow S^1 \times U \\
G &\mapsto (G \circ F^{-1}, ~P(G))
\end{split}
\end{gather}
where $F = \varphi\big(P(G)\big)$. Since $f$ depends smoothly on $G$ and $F$ depends smoothly on $f$, and since inversion and multiplication in the Fr\'echet Lie group $\mathrm{Aut}_1(\xi)$ are smooth maps, we see that $GF^{-1} = \theta$ also depends smoothly on $G$. 
The equations
\begin{gather}
\begin{split}
\Phi \circ \Phi^{-1}(G) &= \Phi(G \circ F^{-1},\, f) = G\circ F^{-1} \circ F = G \\
\Phi^{-1} \circ \Phi(\theta,~f) &= \Phi^{-1}(e^{i\theta}F) = (\theta,~ f)
\end{split}
\end{gather}
confirm that $\Phi$ and $\Phi^{-1}$ are indeed inverses of each other, and this proves that $\Phi$ is a diffeomorphism, so that we have a bundle structure over the open neighborhood $U$ of the identity in $\SDiff(S^2)$.

Finally, the fact that the map $P \colon \mathrm{Aut}_1(\xi) \rightarrow \SDiff(S^2)$ is a smooth homomorphism of Fr\'echet Lie groups provides the homogeneity needed to transfer the above argument to small open sets throughout $\SDiff(S^2)$. This completes our proof that $S^1\to \Aut_1(\xi)\xrightarrow{P}  \SDiff (S^2)$
is a fiber bundle in the world of Fr{\'e}chet manifolds and smooth maps between them.
\end{proof}


\appendix

\section{Fr\'echet spaces and manifolds}\label{appendix_frechet}


\noindent For convenience, we give a brief introduction to Fr\'echet spaces and manifolds in this appendix. After that, we prove some technical results which are used in the proof of the main theorem. For more on this 
subject, we refer the reader to \cite{Ham} and \cite{omori}.

\subsection{Fr\'echet spaces}

\noindent Let $\mathbb{V}$ be a vector space. A \emph{seminorm} on $\mathbb{V}$ is a function $\rho\colon \mathbb{V} \rightarrow [0, \infty)$ satisfying the following properties:
\begin{enumerate}
	\item $\rho(\lambda v) = |\lambda|\rho(v), \forall v \in \mathbb{V}, \lambda \in \mathbb{R}$;
	\item $\rho(v + w) \leq \rho(v) + \rho(w), \forall v, w \in \mathbb{V}$.
\end{enumerate}
If $\rho(v) = 0$ implies $v = 0$, then $\rho$ is called a \emph{norm}.

An arbitrary collection $\{\rho_{\alpha}\}$ of seminorms on $\mathbb{V}$ induces a unique topology $\mathcal{T}$ on $\mathbb{V}$ by declaring that a sequence $\{v_n\}$ in $\mathbb{V}$ converges to $v \in \mathbb{V}$ if and only if $\rho_\alpha(v_n - v) \rightarrow 0$ for all $\alpha$. From this, we declare that a subset $F \subseteq \mathbb{V}$ is closed if it contains its limit points. This topology makes $\mathbb{V}$ into a \emph{topological vector space}, in the sense that the operations of addition and multiplication by scalars are continuous.

Fix a collection $\{\rho_{\alpha}\}$ of seminorms on $\mathbb{V}$ and let $\mathcal{T}$ be the topology generated by them. We say that two collections of seminorms are equivalent if they generate the same topology. Then $\mathcal{T}$ is  \emph{metrizable} if and only if it admits an equivalent \emph{countable} family of seminorms, $\{\rho_{j}\}_{j \in \mathbb{N}}$. In this case, we can define an explicit metric by
\begin{equation}
d(u, v) = \sum\limits_{j = 1}^{\infty} 2^{-j}\, \dfrac{\rho_j(u - v)}{1 + \rho_j(u - v)}
\end{equation}
In this paper, we are interested in the metrizable case, so we work under this assumption from now on. The topology $\mathcal{T}$ is \emph{Hausdorff} if and only if $\rho_{j}(v) = 0$ for all $j$ implies $v = 0$, and it is \emph{complete} if every Cauchy sequence converges. A sequence $\{v_n\}$ in $\mathbb{V}$ is Cauchy if, for each fixed $j$, we have $\rho_j(v_n - v_m) \rightarrow 0$ as $n, m \rightarrow \infty$. 

A vector space $\mathbb{V}$ equipped with a countable family of seminorms $\{\rho_j\}_{j \in \mathbb{N}}$ is a \emph{Fr\'echet space} provided that the topology induced by $\{\rho_j\}_{j \in \mathbb{N}}$, as described above, is Hausdorff and complete.

Let $\mathbb{V}$ and $\mathbb{W}$ be Fr\'echet spaces and $\mathcal{U} \subseteq \mathbb{V}$ be an open set. We say that a continuous map $F\colon \mathcal{U} \subseteq \mathbb{V} \rightarrow \mathbb{W}$ is \emph{differentiable} at $p \in \mathbb{V}$ in the direction $v \in \mathbb{V}$ provided that the limit
\begin{equation}
DF(p)v = \lim\limits_{t \to 0} \dfrac{F(p + tv) - F(p)}{t}
\end{equation}
exists. If this limit exists for all $p \in \mathcal{U}$ and all $v \in \mathbb{V}$, we can form the map
\begin{gather}
\begin{split}
dF\colon \mathcal{U} &\times \mathbb{V}\rightarrow \mathbb{W} \\
(p, v) &\mapsto dF(p)v
\end{split}
\end{gather}
If $dF$ is continuous, as a map from $\mathcal{U} \times \mathbb{V}$ with the product topology into $\mathbb{W}$, then we say $F$ is $C^1$ or \emph{continuously differentiable}. We avoid thinking of $F$ as a map into $\mathrm{L}(\mathbb{V}, \mathbb{W})$, since this is usually not a Fr\'echet space in a natural way. This definition is weaker than the one usually given for maps between Banach spaces.

Proceeding inductively, we define the second derivative of $F$ as
\begin{equation}
d^2F(p)(v_1, v_2) = \lim\limits_{t \to 0} \dfrac{dF(p + tv_1)(v_2) - dF(p)(v_2)}{t}
\end{equation}
and say that $F$ is $C^2$ provided that the map
\begin{gather}
\begin{split}
d^2F\colon \mathcal{U} &\times \mathbb{V} \times \mathbb{V}  \rightarrow \mathbb{W} \\
(p, v_1, v_2) &\mapsto d^2F(p)(v_1, v_2)
\end{split}
\end{gather}
exists and is continuous, and likewise for $C^k$.

We say that $F$ is \emph{smooth} provided it is $C^k$ for all $k$. This notion of smoothness agrees with the standard one in the case where $\mathbb{V}$ and $\mathbb{W}$ are finite dimensional.

A standard example of a Fr\'echet space is $C^{\infty}[a, b]$, the set of all smooth functions from $[a, b]$ to $\mathbb{R}$, equipped with the family of seminorms given by
\begin{equation}
\rho_j(f) = \sup\limits_{x \in [a, b]} |D^jf(x)|
\end{equation}
for $j \geq 0$, with the convention that $D^0f = f$. One can readily check the Hausdorff and completeness conditions.

\subsection{Fr\'echet manifolds}

\noindent A Fr\'echet manifold modeled on $\mathbb{V}$ is a Hausdorff topological space $\mathcal{M}$ with an atlas $\mathcal{A} = \{\varphi_i\}$ of homeomorphisms $\varphi_i \colon U_i \subseteq \mathcal{M} \rightarrow V_i \subseteq \mathbb{V}$ between open sets $U_i$ of $\mathcal{M}$ and $V_i$ of $\mathbb{V}$  such that the transition maps
$$\varphi_j^{-1} \circ \varphi_i \colon U_i \cap U_j \rightarrow U_i \cap U_j$$
are smooth maps between Fr\'echet spaces.

Let $\mathcal{M}$ be a Fr\'echet manifold and $\mathcal{N}$ a closed subset of $\mathcal{M}$. We say that $\mathcal{N}$ is a \emph{Fr\'echet submanifold} of $\mathcal{M}$ if for every $p\in\mathcal{N}$, there exists a coordinate chart $\varphi \colon \mathcal{U} \subseteq \mathcal{M} \rightarrow V \subseteq \mathbb{V}$ of $\mathcal{M}$ with $p \in \mathcal{U}$ and a subspace $\mathbb{W}$ of $\mathbb{V}$ such that
\begin{equation}
\varphi\big(\mathcal{U} \cap \mathcal{N}\big) = \big(\!\{0\} \times \mathbb{W}\big) \cap V
\end{equation}
We say that $\varphi$ is a coordinate chart \emph{adapted} to $\mathcal{N}$. 

At any point $p \in \mathcal{M}$, the \emph{tangent space} $T_p\mathcal{M}$ can be defined as follows. First, consider the set of all triples $(\mathcal{U}, \varphi, v)$, where $\varphi$ is a local chart at $p$ and $v \in \mathbb{V}$. We say that two triples $(\mathcal{U}_i, \varphi_i, v_i)$, $i = 1, 2$, are equivalent if
$$d(\varphi_2 \circ \varphi_1^{-1})v_1 = v_2$$
Then $T_p\mathcal{M}$ is the set of all such equivalence classes. Although this is a rather cumbersome description of the tangent space, in many situations a much more concrete one is available, as we shall see below. In what follows, we describe in detail a number of Fr\'echet manifolds that are used throughout the paper.

\subsection{Examples}

\noindent Let $M$ be a smooth, closed, finite-dimensional manifold. Then the group $\mathrm{Diff}(M)$ of all diffeomorphisms from $M$ to itself, equipped with the $C^{\infty}$ topology, is a Fr\'echet manifold. Following  \cite{Eells}, we describe an atlas for $\mathrm{Diff}(M)$, modeled on Fr\'echet spaces of vector fields.

Let $C^{\infty}(TM)$ be the space of all smooth vector fields on $M$. Choose a Riemannian metric $g$ on $M$ and let $\nabla$ denote its Levi-Civita connection. For each $n \in \mathbb{N}$, let
\begin{equation}
\| v \|_n = \sup\limits_{x \in M} \big\|(\nabla^n v)(x)\big\|
\end{equation}
where
\begin{equation}
\big\|(\nabla^n v)(x)\big\| = \sup_{\substack{\| e_i \| = 1 \\ i=1,\ldots,n}} \big\| \nabla_{e_1}\cdots\nabla_{e_n} v(x)\big\|
\end{equation}
The vector space $C^{\infty}(TM)$ equipped with the collection of seminorms $\{ \| \, \|_n \}$ is a Fr\'echet space (cf \cite{Ham}). More generally, given $f \in \mathrm{Diff}(M)$, we let
\begin{equation}
C^{\infty}(f^\dast TM) = \{ v \circ f \, : \, v \in C^{\infty}(TM) \}
\end{equation}
The set $C^{\infty}(f^\dast TM)$ of vector fields along $f$ is again a Fr\'echet space, and the map $v \mapsto v \circ f$ is a linear isomorphism between $C^{\infty}(TM)$ and $C^{\infty}(f^\dast TM)$.

Let $\exp_p \colon T_pM \rightarrow M$ be the exponential map associated with the Riemannian metric $g$ on $M$. Given a diffeomorphism $f \in \mathrm{Diff}(M)$, there exists an open neighborhood $\mathcal{U}_f \subseteq C^{\infty}(f^\dast TM)$ containing the zero section, and an open neighborhood $V_f \subseteq \mathrm{Diff}(M)$ containing $f$ such that
\begin{gather}\label{exp_chart}
\begin{split}
\mathrm{Exp}_f \colon U_f \subseteq C^{\infty}(f^\dast TM) &\rightarrow V_f \subseteq \mathrm{Diff}(M) \\
v \circ f &\mapsto \exp \big( v \circ f \big)
\end{split}
\end{gather}
is a homeomorphism (\cite{leslie},\! \cite{omori},\! \cite{KM}). We see from the definition that the transition maps are smooth. The collection of maps $\{ \mathrm{Exp}_f : f \in \mathrm{Diff}(M) \}$ cover $\mathrm{Diff}(M)$, and the maximal atlas compatible with this collection defines the manifold structure on $\mathrm{Diff}(M)$. Furthermore, this manifold structure makes $\mathrm{Diff}(M)$ a \emph{Fr\'echet Lie group}, in the sense that the natural operations of multiplication
\begin{gather}
\begin{split}
\circ : \mathrm{Diff}(M) \times \mathrm{Diff}(M) &\rightarrow \mathrm{Diff}(M) \\
(f, \, g) &\mapsto f \circ g
\end{split}
\end{gather}
and inversion
\begin{gather}
\begin{split}
\mathrm{inv} : \mathrm{Diff}(M) &\rightarrow \mathrm{Diff}(M) \\
f &\mapsto f^{-1}
\end{split}
\end{gather}
are smooth. We remark that it is possible to model $\mathrm{Diff}(M)$ as a Banach manifold, if we choose to work with the $C^k$ topology, or a Hilbert manifold, using $L^2$ Sobolev topologies. In this case, we could construct coordinate charts in the same way as \eqref{exp_chart}. However, the resulting Banach or Hilbert manifold would \emph{not} be a Lie group: both the composition and the inversion maps above would be continuous but not differentiable.

On the other hand, a disadvantage of working in the Fr\'echet category, as opposed to the Banach or Hilbert category, is that the classical Inverse Function Theorem is no longer true. Instead, it must be replaced by the celebrated \emph{Nash-Moser Inverse Function Theorem}; see \cite{Ham} for a detailed account of this. We will not need this theorem here.

The propositions to follow, \autoref{prop_path_sn} through \autoref{prop_smoothness_F}, are there to help us prove that the diffeomorphism $F \in \mathrm{Aut}_1(\xi)$ from \autoref{lemma_5_8} and \autoref{fiberbundle} depends smoothly on the point $x \in S^3$, the path $\gamma \in \mathrm{Path}(S^2)$ and the diffeomorphism $f \in \SDiff(S^2)$, the ingredients which went into its construction.

\begin{prop}\label{prop_path_sn}
The space $\mathrm{Path}(S^n)$  of $C^\infty$ maps from the interval $[0,1]$ into $S^n$ is a Fr\'echet manifold.
\end{prop}
\begin{proof} Fix a curve $\gamma \in \mathrm{Path}(S^n)$. Then we can parametrize nearby curves in $\mathrm{Path}(S^n)$ by the Fr\'echet space
\begin{gather*}
\begin{split}
T_{\gamma}\mathrm{Path}(S^n) = \{ V\colon [0, 1] \rightarrow TS^n \, : \, \pi \circ V = \gamma \}
\end{split}
\end{gather*}
of vector fields on $S^n$ along $\gamma$, where $\pi : TS^n \rightarrow S^n$ is the projection from the tangent bundle of $S^n$ to $S^n$. The correspondence between these vector fields and curves near $\gamma$ is given by the Riemannian exponential map
\begin{gather*}
\begin{split}
\mathrm{Exp}\colon \mathcal{U} \subset T_{\gamma}\mathrm{Path}(S^n) &\rightarrow \mathrm{Path}(S^n) \\
V &\mapsto \mathrm{Exp}_{\gamma}(V)
\end{split}
\end{gather*}
where $\mathcal{U}$ is the subset of vector fields along $\gamma$ with magnitude less than $\pi/2$. 

The inverse of this map is given as follows. If $\beta \in \mathrm{Path}(S^n)$ is a curve close to $\gamma$, meaning that the spherical distance $d_{S^n}(\beta(t), \gamma(t)) < \pi/2$ for all $t$, then there exists a unique geodesic from $\gamma(t)$ to $\beta(t)$ with initial velocity $V(t)$. By construction,
$$\mathrm{Exp}_{\gamma(t)}(V(t)) = \beta(t),~~\forall t \in [0, 1]$$
This proves that $\mathrm{Path}(S^n)$ is a Fr\'echet manifold (cf \cite[Example 4.2.3]{Ham}).\end{proof}

	\begin{figure}[h!]
	\begin{center}
		\includegraphics[scale=0.45]{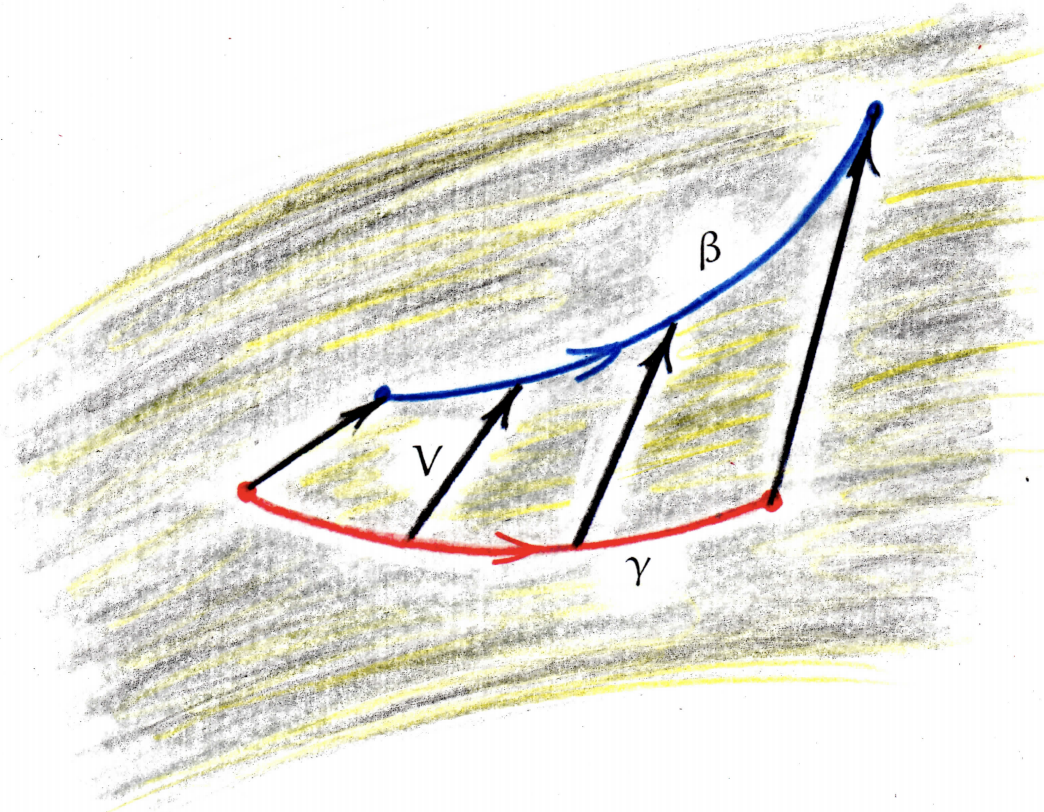}\caption{$\beta=\mathrm{Exp}_\gamma(V)$}\label{beta}
	\end{center}
	\end{figure}

\begin{prop}\label{prop_path_star}
	The set
	$$\mathrm{Path}^\dast = \{ (\gamma, x) \, : \text{$\gamma \in \mathrm{Path}(S^2)$ and $x \in S^3$ with $\gamma(0) = p(x) \in S^2$} \, \}$$
	is a Fr\'echet manifold, and a smooth submanifold of $\mathrm{Path}(S^2) \times S^3$.
\end{prop}
\begin{proof} Fix a point $(\gamma, x) \in \mathrm{Path}^\dast$. We will show that points in $\mathrm{Path}^\dast$ near $(\gamma ,x)$ can be parametrized by vectors in the Fr\'echet space
\begin{equation*}
T_{(\gamma,\, x)}\mathrm{Path}^\dast = \big \{ (V, w) \, : \, V \in T_{\gamma}\mathrm{Path}(S^2),~w \in T_xS^3~\text{and}~ V(0) = dp(x)w \big \}
\end{equation*}
Choose a local trivialization of the Hopf fibration
\begin{equation*}
\Psi \colon U_0 \times S^1 \rightarrow p^{-1}(U_0)
\end{equation*}
containing $p(x) \in U_0$. Using this trivialization, for each $y \in p^{-1}(U_0)$ we write
\begin{equation}\label{eq_decomp_tangent}
T_yS^3 = T_{p(y)}S^2 \oplus \mathbb{R}.
\end{equation}
\begin{figure}[h!]
	\begin{center}
		\includegraphics[scale=0.45]{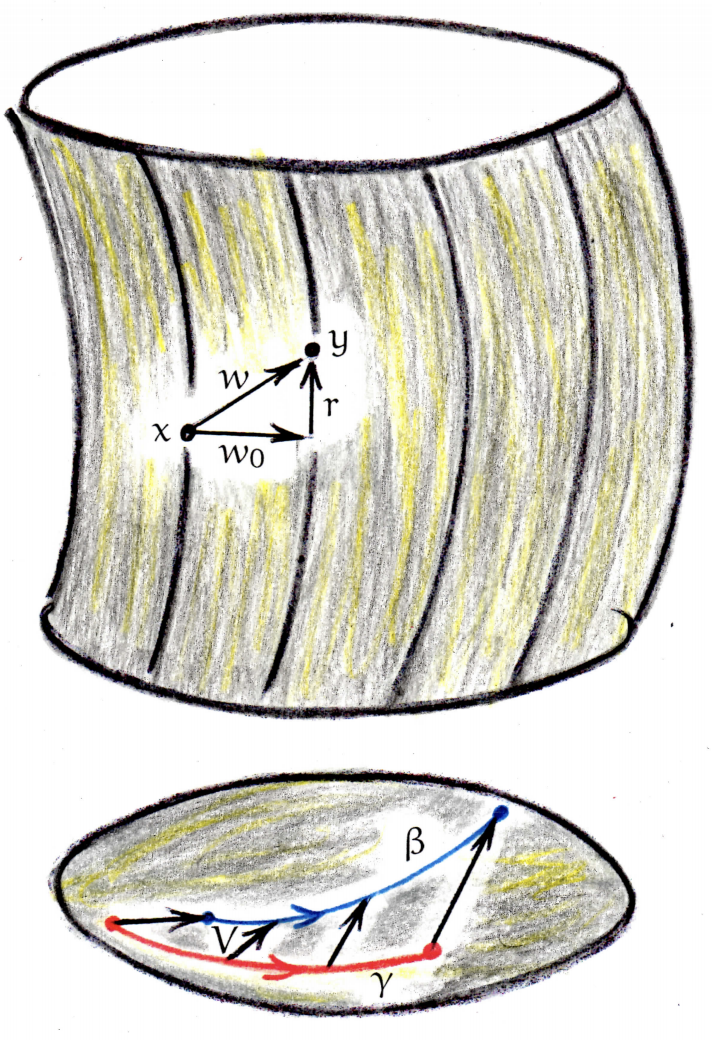}\caption{$\mathrm{Path}^\dast$ is a Fr{\'e}chet manifold}\label{pathstar}
	\end{center}
	\end{figure}	

Now, given $(V, w) \in T_{(\gamma,\, x)}\mathrm{Path}^\dast$ with $V$ and $w$ sufficiently small, we first let
\begin{equation*}
\beta(t) = \mathrm{Exp}_{\gamma(t)}\big(V(t)\,\big)
\end{equation*}
as before, so $\beta$ is a curve in $S^2$ near the original $\gamma$. Then write $w = (w_0, r)$ according to the decomposition \autoref{eq_decomp_tangent}, and set
\begin{equation*}
y = \Psi\big(\mathrm{Exp}(\! w_0),\,\, e^{i r}\big)
\end{equation*}
where  $\mathrm{Exp}$ is the exponential map in $S^3$. The geodesic $s \mapsto \mathrm{Exp}(s\, w_0)$ is horizontal to the Hopf fibers, since it starts that way and $p$ is a Riemannian submersion.
The map
\begin{equation}
\widetilde{\mathrm{Exp}}(V, w) = (\beta, y)
\end{equation}
is our coordinate chart for $\mathrm{Path}^\dast$. It is clear that any pair $(\beta, y) \in \mathrm{Path}^\dast$ sufficiently close to $(\gamma, x)$ can be obtained in this way as the image of some $(V, w)$ under $\widetilde{\mathrm{Exp}}$. 
	
	\end{proof}

\vskip 0.1in
\begin{prop}\label{prop_lift_smooth}
The map $\mathrm{Lift} : \mathrm{Path}^\dast \rightarrow \mathrm{Path}(S^3)$, which takes a pair $(\gamma, x)$ to the unique horizontal lift $\overline{\gamma}$ of $\gamma$ starting at $x$, is smooth.
\end{prop}
\begin{proof} Let $\overline{\gamma}(t) = \mathrm{Lift}(\gamma, x)(t)$. By definition, $\overline{\gamma}$ is the unique solution of the system
\begin{gather}\label{eq_lift_definition}
\begin{split}
\langle \overline{\gamma} \ ', A \rangle &= 0 \\
p \circ \overline{\gamma}(t) &= \gamma(t) \\
\overline{\gamma}(0) &= x
\end{split}
\end{gather}
 which depends smoothly on the initial condition $x$ and the parameter $\gamma$. We will  compute this dependence explicitly when lifting curves from $\SDiff(S^2)$ to $\mathrm{Aut}_1(\xi)$.
 
  \end{proof}
 	\begin{figure}[h!]
	\begin{center}
		\includegraphics[scale=0.45]{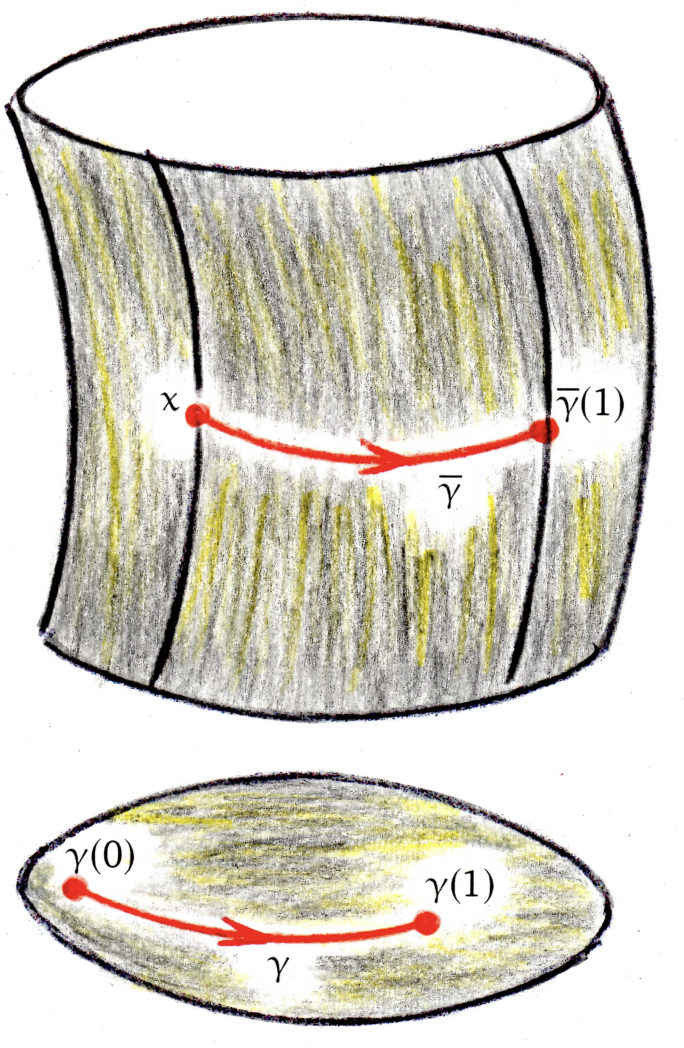}\caption{The path $\gamma$ in $S^2$ lifts to the horizontal path $\overline{\gamma}$ in $S^3$}\label{liftpathstar}
	\end{center}
	\end{figure}	
	
\vspace{10pt}	

Borrowing notation from the proof of \autoref{prop_lift_smooth}, we let $\mathrm{Eval} \colon \mathrm{Path}^\dast \rightarrow S^3$ be the map that sends $(\gamma, x)$ to the endpoint $\overline{\gamma}(1)$ of its lift. Then this is also a smooth map.

\begin{prop}\label{prop_eval_smooth}
The map $\mathrm{Eval} \colon \mathrm{Path}^\dast \rightarrow S^3$ is smooth.
\end{prop}
\begin{proof} Note that $\mathrm{Eval}(\gamma, x) = \mathrm{E_1} \circ \mathrm{Lift}(\gamma, x)$, where
\begin{gather*}
\begin{split}
\mathrm{E_1} \colon \mathrm{Path}(S^3) &\rightarrow S^3 \\
c &\mapsto c(1)
\end{split}
\end{gather*}
The map $\mathrm{E_1}$ is smooth: its first derivative at any $\alpha$ is
\begin{gather*}
\begin{split}
\mathrm{dE_1}(\alpha) \colon T_{\alpha}\mathrm{Path}(S^3) &\rightarrow T_{\alpha(1)}S^3 \\
V &\mapsto V(1)
\end{split}
\end{gather*}
which is a bounded map between Fr\'echet spaces. The same remark applies for higher derivatives. Since $\mathrm{Eval}$ is a composition of smooth maps, it is also smooth by the Chain rule. \end{proof}

We now turn to our main goal in this appendix, which is to prove explicitly that the map $F \colon S^3 \rightarrow S^3$ defined in \autoref{eq_definition_F} is smooth. Recall that to define this map, we first fix a point $y_0 \in S^2$ and a rigid motion
$$F_0 \colon p^{-1}(y_0) \rightarrow p^{-1}(f(y_0))$$
between the Hopf fibers $p^{-1}(y_0)$ and $p^{-1}(f(y_0))$, where $f \colon S^2 \rightarrow S^2$ is a given area-preserving diffeomorphism. Then, $F$ is given by the composition
\begin{equation}\label{eq_F_appendix}
F(x) = H_{f\gamma} \circ F_0 \circ H_{\gamma}^{-1}(x)
\end{equation}
where $\gamma$ is any path in $S^2$ between $y_0$ and $y = p(x)$ and the maps $H$ are the horizontal transport maps defined in 
\autoref{eq_horizontal_transport}.

\begin{prop}\label{prop_smoothness_F}
	The map $F$ is smooth as a function of the point $x \in S^3$, the
	path $\gamma \in \mathrm{Path}(S^2)$ and the diffeomorphism $f \in \SDiff(S^2)$.
\end{prop}

\begin{proof}
	It suffices to check that each of the factors in \autoref{eq_F_appendix} is smooth. To do that, we first focus on the points $x \in S^3$ with $y = p(x)$ close to the base point $y_0$. Given such an $x$, choose $\gamma_y$ to be the unique shortest geodesic between $y$ and $y_0$. Then $\gamma_y$ depends smoothly on $y$ and
	\begin{equation}
	H_{\gamma_y}^{-1}(x) = \mathrm{Eval}(\gamma^{-1}_y,~ x)
	\end{equation}
	in turn depends smoothly on $x$, by  \autoref{prop_lift_smooth} and \autoref{prop_eval_smooth}. Similarly,
	\begin{equation}
	H_{fy} = \mathrm{Eval}(f \circ \gamma,~ F_0(x_0))
	\end{equation}
	and since $F_0$ is a fixed rigid motion, it follows that $F$ is smooth, at least on a neighborhood of the fiber $p^{-1}(y_0)$. To treat the case where $x$ is far away from this fiber, it suffices to note that we can choose a different base point $y_0$ whose fiber $p^{-1}(y_0)$ is close to $x$, since this new choice of base point will yield the same map $F$ up to a uniform rotation of all fibers. Thus, $F$ is everywhere smooth.
\end{proof}

\newpage

\bibliographystyle{amsalpha}
\bibliography{../biblioSF}

\vspace{40pt}

\begingroup%
\setlength{\parskip}{\storeparskip}

\end{document}